\documentclass[a4paper,reqno]{amsart}

\usepackage[foot]{amsaddr}
\usepackage{graphicx,enumerate,nicefrac,color,bm,cite}
\usepackage{amsmath, amssymb, amstext, amsfonts}
\usepackage{stmaryrd}
\usepackage{geometry}
\usepackage{subfig}
\usepackage{algorithm,algpseudocode,tabularx}

\makeatletter
\newcommand{\multiline}[1]{%
  \begin{tabularx}{\dimexpr\linewidth-\ALG@thistlm}[t]{@{}X@{}}
    #1
  \end{tabularx}
}
\newcommand{\algmargin}{\the\ALG@thistlm}
\makeatother

\algnewcommand{\parState}[1]{\State%
  \begin{minipage}[t]{\dimexpr\linewidth-\algorithmicindent-\algorithmicindent}{#1\strut}\end{minipage}}
\newcommand{\myState}[1]{\State\parbox[t]{\dimexpr\linewidth-\algorithmicindent}{#1\strut}}

\newcommand{\norm}[1]{\|#1\|}
\newcommand{\Norm}[1]{\left\|#1\right\|}
\newcommand{\nnn}[1]{|\!|\!|#1|\!|\!|}

\newcommand{\operator}[1]{\mathsf{#1}}

\newcommand{\F}{\operator{F}}   
\renewcommand{\d}{\operator{d}}
\newcommand{\ds}{\,\d s}

\newcommand{\dx}{\,\d x}
\newcommand{\jmp}[1]{\left\llbracket#1\right\rrbracket}

\newcommand{\dprod}[1]{\left<#1\right>_{{\H}^\star\times {\H}}}

\newcommand{\Je}{\J_\varepsilon}

\newcommand{\sL}{S_{L^\perp}}
\newcommand{\C}{\mathrm{C}}
\newcommand{\J}{\mathsf{E}}
\newcommand{\hsp}[2]{#1_{#2}}
\newcommand{\MX}{\mathcal{P}(\hsp{L}{v})}
\newcommand{\ML}{\H}
\newcommand{\cU}{\kappa_\nu}
\newcommand{\R}{\mathsf{R}}
\newcommand{\g}{\mathsf{g}}
\newcommand{\p}{\mathfrak{p}}
\newcommand{\q}{\mathfrak{q}}
\renewcommand{\r}{\mathfrak{r}}

\newcommand{\E}{\J}
\renewcommand{\H}{\mathbb{X}}
\newcommand{\Ee}{\E_\varepsilon}
\newcommand{\HS}{\mathrm{H}_0^1} 
\newcommand{\X}{\mathbb{X}}
\renewcommand{\L}{\mathrm{L}}
\renewcommand{\P}{\mathbb{P}}
\renewcommand{\vec}[1]{\bm{\mathsf{#1}}}

\DeclareMathOperator*{\argmax}{arg\,max}

\DeclareMathOperator*{\spn}{span}
\DeclareMathOperator*{\imag}{im}

\newtheorem{theorem}{Theorem}[section]

\newtheorem{proposition}[theorem]{Proposition} 
\newtheorem{cor}[theorem]{Corollary}

\theoremstyle{definition}

\newtheorem{remark}[theorem]{Remark}

\title[Local Minimax Galerkin Methods]{Adaptive local minimax Galerkin methods\\ for variational problems}

\author[P.~Heid \and T.~P.~Wihler]{Pascal Heid \and Thomas P.~Wihler}
\address{Mathematics Institute, University of Bern, CH-3012 Switzerland}
\email{pascal.heid@math.unibe.ch \and wihler@math.unibe.ch}

\thanks{The authors acknowledge the financial support of the Swiss National Science Foundation under grant no. 200021\underline{\phantom{x}}182524}



\begin{document}

\begin{abstract}
In many applications of practical interest, solutions of partial differential equation models arise as critical points of an underlying (energy) functional. If such solutions are saddle points, rather than being maxima or minima, then the theoretical framework is non-standard, and the development of suitable numerical approximation procedures turns out to be highly challenging. In this paper, our aim is to present an iterative discretization methodology for the numerical solution of nonlinear variational problems with multiple (saddle point) solutions. In contrast to traditional numerical approximation schemes, which typically fail in such situations, the key idea of the current work is to employ a simultaneous interplay of a previously developed local minimax approach and adaptive Galerkin discretizations. We thereby derive an adaptive \emph{local minimax Galerkin (LMMG)} method, which combines the search for saddle point solutions and their approximation in finite-dimensional spaces in a highly effective way. Under certain assumptions, we will prove that the generated sequence of approximate solutions converges to the solution set of the variational problem. This general framework will be applied to the specific context of finite element discretizations of (singularly perturbed) semilinear elliptic boundary value problems, and a series of numerical experiments will be presented. 
\end{abstract}

\maketitle

\section{Introduction}

Consider a (real) Hilbert space $\H$, equipped with an inner product $(\cdot,\cdot)_{\H}$ and an induced norm~$\norm{\cdot}_{\H}$. Given a nonlinear operator $\F:\,\H \to \H^\star$, where ${\H}^\star$ signifies the dual space of ${\H}$, we focus on solutions of the equation
\begin{align} \label{eq:F=0}
 u \in {\H}: \qquad \F(u)=0 \quad \text{in } {\H}^\star.
\end{align}
This problem is variational if there exists an underlying functional $\J \in \C^1({\H};\mathbb{R})$ such that solutions $u\in {\H}$ of \eqref{eq:F=0} arise as critical points of~$\J$, i.e. if they satisfy the Euler--Lagrange equation
\begin{align} \label{eq:J'=0}
 u \in {\H}: \qquad \J'(u)=0 \quad \text{in } {\H}^\star,
\end{align}
with $\J'$ denoting the Fr\'{e}chet derivative of $\J$. A solution to the Euler--Lagrange equation~\eqref{eq:J'=0} is called a \emph{critical point}, and the value of the  functional $\J$ at a critical point is termed \emph{critical value}. 

The purpose of this paper is to provide a new adaptive algorithm for the numerical solution of~\eqref{eq:J'=0}, which exploits the theoretical framework of the mountain pass critical point theory, in combination with adaptive Galerkin discretizations. The key idea is to exploit an automatic and simultaneous interplay of these two approaches in order to design a highly effective numerical approximation procedure. This will be illustrated in the specific context of classical adaptive finite element discretizations of singularly perturbed semilinear partial differential equations (PDE); such problems have wide ranging applications in practice (including, e.g., nonlinear reaction-diffusion in ecology and chemical models~\cite{CaCo03,Ed05,Fr08,Ni11,OkLe01}, economy~\cite{BaBu95}, or classical and quantum physics~\cite{BeLi83,St77}). Yet, they are notoriously challenging to solve numerically due to the existence of several (or even infinitely many) solutions and/or the appearance of singular effects including boundary layers and (multiple) spikes.

\subsection*{Minimax theory}
Basic critical point theory in the calculus of variations pays attention to critical points that are either local minima or maxima of a given functional~$\E\in\C^1(\X;\mathbb{R})$. Many solutions to nonlinear variational problems of practical relevance, however, occur as unstable critical points, i.e. they are neither a (local) maximum nor minimum. Such unstable critical points are called \emph{saddle points}: More precisely, a saddle point of a functional $\J$ is an element $u^\star \in {\H}$ such that $\J'(u^\star)=0$, and for any (open) neighbourhood $\mathcal{U}(u^\star)$ of $u^\star$ there are $u,v \in \mathcal{U}(u^\star)$ with
\[
 \J(u)<\J(u^\star)<\J(v).
\]

In the minimax theory of Ambrosetti--Rabinowitz~\cite{AmRa:73}, see also~\cite[\S2]{Rabinowitz:86}, saddle points appear as solution to a two-level optimization problem of the form
\[
 \min_{A \in \mathcal{A}} \max_{u \in A} \J(u),
\]
where $\mathcal{A}$ is a collection of subsets of ${\H}$. In this context, a central result for the existence of (multiple) critical points, especially of saddle points, is the mountain pass theorem. It is based on the so-called \emph{Palais--Smale compactness condition} of the functional $\J$:
\begin{enumerate}[(PS)]
\item Any sequence $\{u^k\}_k \subset {\H}$ for which $\{\J(u^k)\}_k$ is a bounded sequence in $\mathbb{R}$, and $\J'(u^k) \to 0$ in ${\H}^\star$, as $k \to \infty$, possesses a convergent subsequence. 
\end{enumerate}

\begin{theorem}[Mountain pass theorem] \label{thm:mountainpass}
Let ${\H}$ be a (real) Hilbert space and $\J \in \C^1({\H};\mathbb{R})$ satisfying the Palais--Smale condition {\rm (PS)}. Suppose that 
 \begin{enumerate}[\rm (a)]
 \item $\J(0)=0$;
  \item there are constants $r,\alpha >0$ such that $\J(v) \geq \alpha$ for all $v\in\X$ with $\norm{v}_{\X}=r$;
  \item there exists an element $h \in {\H}$ with $\norm{h}_{\H}>r$ such that $\J(h) \leq 0$.
 \end{enumerate}
Then, $\J$ possesses a critical value $c \geq \alpha$, which can be expressed by
\[
 c=\inf_{g \in \Gamma} \max_{u \in g([0,1])} \J(u),
\]
where $\Gamma=\{g \in \C^0([0,1];{\H}):\,g(0)=0, \ g(1)=h\}$.
\end{theorem}

\subsection*{Mountain pass type numerical algorithms}
The importance of saddle points appearing in natural science applications has raised a high demand for non-standard numerical approaches for nonlinear variational problems. Although this endeavour turns out to be tremendously challenging in practice, it seems natural to apply the framework provided by the minimax theory, and, in particular, to exploit the analytical foundation of the mountain pass theorem for the purpose of designing numerical solution algorithms for~\eqref{eq:J'=0}. This route was first pursued by Choi and McKenna in their pathbreaking paper \cite{choimckenna:93}. In conjunction with its theoretical counterpart, their scheme is widely known as the \emph{Mountain Pass Algorithm (MPA)}. The core of the method is an iterative steepest descent procedure which takes care of finding a minimum along a local mountain range of $\J$. This part of the algorithm is expressed in terms of a \emph{linear equation}; in the context of nonlinear PDE, for example, this problem can be solved by traditional numerical discretization schemes such as the finite element method (FEM).

Although the work of Choi and McKenna can certainly be considered a milestone in the development of numerical solution schemes for nonlinear variational problems, many issues have remained open. For instance, their paper \cite{choimckenna:93} does not contain any error or convergence analysis of the proposed MPA. Moreover, the MPA will typically find critical points of Morse index 0 or 1 only. Several subsequent papers have used the MPA approach in order to achieve further progress on the topic: In the article~\cite{dingcostachen:99}, for example, a numerical algorithm to compute sign-changing solutions of the semilinear elliptic PDE 
\begin{align} \label{eq:semilinearelliptic1}
 -\Delta u = f(\cdot,u),
\end{align}
subject to Dirichlet boundary conditions, was proposed. The main idea is to construct a \emph{local link} from a known critical point to a new critical point. The resulting \emph{high-linking algorithm (HLA)} presumes that a mountain pass solution is already available, and the MPA \cite{choimckenna:93} is applied as part of the scheme. This work showed that the HLA is able to generate sign-changing solutions for non-symmetric domains and odd nonlinearities, which could not be found by the original MPA. Moreover, in the special case of symmetric domains, the paper~\cite{CoDiNe2001} has focused on sign-changing solutions of \eqref{eq:semilinearelliptic1}. The approach is based on a \emph{modified Mountain Pass Algorithm (MMPA)}, which applies a restriction of the underlying functional to the fixed point set of certain compact topological groups representing the symmetry of the domain. Still, no convergence analysis was done by then.

\subsection*{Local minimax approach}
Further progress in the development of numerical schemes for nonlinear variational problems with multiple solutions was made by Li and Zhou in their paper~\cite{LiZhou:01}. Beginning with a set of already known critical points, the idea is to use solution-submanifolds of so-called peak selection mappings, whose local minima occur as new unstable critical points of the underlying functional. One of the crucial advantages of the \emph{local minimax (LMM) algorithm} proposed by Li and Zhou is that the generated sequence of approximated solutions exhibits a decay of the associated energy functional. Moreover, the LMM algorithm may find critical points of Morse index greater than 1. In their subsequent paper~\cite{LiZhou:02}, Li and Zhou have introduced a new step-size rule for the LMM procedure leading to the \emph{modified local minimax algorithm}. Moreover, a first convergence analysis has been derived; remarkably, under certain assumptions, the LMM algorithm is able to generate approximation sequences which contain converging subsequences to a critical point of the energy functional. In addition, for initial guesses sufficiently close to an isolated critical point, the authors have proved that the iteration converges precisely to that point. This result was improved further by Zhou in \cite[Thm.~2.4]{Zhou:17}; we note that the proof of that result can be adapted in such a way that it yields the convergence of the generated sequence to the set of critical points. 

The original local minimax approach by Li and Zhou~\cite{LiZhou:01,LiZhou:02} has been studied, modified, generalized, and applied in various subsequent papers by Zhou, and other authors. We mention the work~\cite{XieYuanZhou:12} where a modified LMM method to find multiple solutions of singularly perturbed semilinear PDE with Neumann boundary conditions has been presented. In that paper, the authors have proposed an 'ad hoc' computational idea on how local mesh refinements in the framework of finite element Galerkin spaces, with a special emphasis on the resolution of spike layers, may be applied. Specifically, after a (fixed) number of steps in the LMM, mesh refinements are performed whenever the residual is not sufficiently small; the refinements, in turn, are performed simply by subdividing any elements where the numerical solution tends to form a spike. A further article~\cite{Yao:15} has paid special attention to the convergence analysis of the LMM scheme on a finite dimensional Galerkin space. Within this setting, it has been proved that the generated discrete sequence possesses subsequences that converge to a solution of the discrete problem. Furthermore, for finite element discretizations of semilinear elliptic equations, as the mesh size tends to zero, it has been shown that the subset of discrete solutions which can be approximated by the LMM algorithm converges to a subset of solutions of the original problem. This does not, however, imply the convergence of the generated sequence to a solution of the problem. Finally, we point to the work \cite{WangZhou:05} which offers a modification of the LMM method based on applying a projection onto subspaces with certain symmetry properties; this, in turn, allows to find saddle points with corresponding symmetric features (including critical points of higher Morse index). 

\subsection*{Contribution} Whilst most previous works on local minimax methods focus on  abstract (not necessarily finite dimensional) spaces (and are thus not feasible in actual simulation practice), the aim of the current work is to provide a \emph{computational} approximation procedure for saddle points of a functional $\J$, which is based on a simultaneous interplay of the LMM approach proposed in~\cite{LiZhou:01} and efficient adaptive Galerkin space enrichments. This idea follows the recent developments on the (adaptive) \emph{iterative linearized Galerkin (ILG)} methodology~\cite{HeidWihler:18,HeidWihler:19,CongreveWihler:17,AmreinWihler:14,AmreinWihler:15,HoustonWihler:18}, whereby adaptive discretizations and iterative linearization solvers are combined in an intertwined way; we also refer to the closely related works~\cite{GantnerHaberlPraetoriusStiftner:17,ErnVohralik:13,El-AlaouiErnVohralik:11,BernardiDakroubMansourSayah:15,GarauMorinZuppa:11,HeidStammWihler:19}.

A key building block of the numerical scheme to be presented in this paper concerns the decision of whether hierarchical Galerkin space enrichments or LMM iterations on the current discrete space should be given preference. This is accomplished by estimating the residual on a given Galerkin space in terms of a computable indicator. Once the residual is found sufficiently small, we conclude that any further LMM iteration will not significantly reduce the residual on the present discrete space. Consequently, by making use of local residual indicators, we will hierarchically enrich the Galerkin space. On the theoretical side, under certain assumptions, we prove that the sequence generated by the adaptive LMM Galerkin (LMMG) algorithm converges to the set of critical points of $\J$. 

Special attention will be given to (singularly perturbed) semilinear elliptic PDE in the context of standard finite element discretizations. We will apply the approach presented in~\cite{Verfurth:13}, which has been developed for linear elliptic problems, in order to derive a posteriori residual bounds for the LMMG algorithm that are robust with respect to the singular perturbation parameter; see also~\cite{AmreinWihler:15} for related results in the context of Newton-type linearizations of semilinear singularly perturbed PDE. Our numerical tests display optimal convergence rates with respect to the number of elements in the mesh.

\subsection*{Outline of the paper} 
We will briefly recall the main concepts of the local minimax algorithm from Li and Zhou in \S\ref{sc:LZ}, and present some summarized results from~\cite{LiZhou:01,LiZhou:02,Yao:15}. Furthermore, the focus of \S\ref{sc:LMMG} is on the new (abstract) adaptive LMMG algorithm that exploits an interplay between the classical LMM method and adaptively enriched general Galerkin discretizations. In addition, under certain assumptions, we will prove that the approximated discrete solution sequence generated by the proposed LMMG procedure converges to the set of solutions of the original problem. Then, in \S\ref{sc:SL}, we will show that our general theory applies to a class of singularly perturbed semilinear elliptic boundary value problems. A series of numerical experiments will be presented in \S\ref{sec:numexp}. 

\section{The local minimax approach by Li and Zhou}\label{sc:LZ}

In this section, we revisit the local minimax method introduced in~\cite{LiZhou:01,LiZhou:02}. For the convenience of the reader, we will recall the relevant definitions, and point to some existing results.  In the sequel, let $\J \in \C^1({\H};\mathbb{R})$ be a given functional that satisfies the Palais--Smale compactness condition~{(PS)}.

\subsection{Peak selection}

Let $L \subset {\H}$ be any \emph{closed} subspace, and denote by $L^\perp$ its orthogonal complement with respect to the ${\H}$-inner product: $L \oplus L^\perp={\H}$. Then, for any $v$ in the sphere $S_{L^\perp}:=\{v \in L^\perp:\norm{v}_{\H}=1\}$, we define the closed half space
\[
\hsp{L}{v}:=L+\{tv:\,t\ge 0\}=\{u+tv: u \in L, \ t \geq 0\}.
\] 
For given $v\in\sL$, a point $w_0 \in\hsp{L}{v}$ is called a local maximum of $\J$ in $\hsp{L}{v}$ if there exists $\delta>0$ such that $\J(w)\le\J(w_0)$ for all $w\in\hsp{L}{v}$ with $\norm{w-w_0}_{\H}<\delta$; the set of all local maxima of $\J$ in $\hsp{L}{v}$ is signified by $\MX$. 

A single-valued mapping $p:\,\sL \to {\H}$ with $p(v) \in \MX$ for all $v \in \sL$ is called a \emph{peak selection of $\J$} (with respect to~$L$). Furthermore, for a point $x_0 \in \sL$, we say that $\J$ has a \emph{local peak selection} at $x_0$ (with respect to~$L$) if there exists $\delta>0$ and a mapping $p:\,\{v\in\sL:\,\norm{v-x_0}_{\H}<\delta\} \to \ML$ with $p(v) \in \MX$ for all $v$ in the domain of the local mapping $p$.

The following observation about local peak selections is instrumental in view of an algorithmic development of the mountain pass theory.

\begin{proposition}[Theorem~2.1 of~\cite{LiZhou:01}]\label{pr:LZ}
Suppose that $\J$ has a local peak selection $p$ at some point $x_0\in\sL$ (with respect to $L$). Furthermore, let the following conditions hold true:
\begin{enumerate}[\rm (a)]
\item $p$ is continuous at $x_0$;
\item it holds $\inf_{u\in L}\norm{p(x_0)-u}_{\H} \geq \alpha$ for some constant $\alpha>0$;
\item $x_0$ is a local minimum point of $\J(p(\cdot))$ on $\sL$. 
\end{enumerate}
Then, $p(x_0)$ is a critical point of $\J$.
\end{proposition}

For a given peak selection $p$ of $\J$ (with respect to $L$), we consider the image of $p$ given by
\[
\imag(p):=\{p(v):\,v \in S_{L^\perp}\}\subset \ML.
\]
Then, under the assumptions (a)--(c), the above result shows that a local minimizer of $\J$ in $\imag(p)$ is a critical point of $\J$ in ${\H}$. 

\subsection{Local minimax scheme}\label{sc:minimax}

We will briefly outline the main ideas of the minimax algorithm from~\cite[\S2]{LiZhou:02}. To this end, for $n\ge 0$, consider previously known critical points $w_1,\ldots,w_{n-1}\in {\H}$ ordered in such a way that 
\begin{equation}\label{eq:w1n}
\J(w_1)\le\ldots\le\J(w_{n-1}),
\end{equation} 
and define the linear subspace 
\[
L:=\spn\{w_1,\dotsc,w_{n-1}\}.
\]
We aim to find a new critical point $w_n\in\X$ by pursuing the following procedure:
\begin{enumerate}[(i)]
\item For a given $v^0 \in \sL$, we suppose that there is $t^0>0$ such that
\begin{align} \label{eq:ascent1}
 w^0:=p(v^0)=u^0+t^0 v^0\in L_{v^0},
\end{align}
for some $u^0 \in L$.
\item Now we find a local minimum of $\J$ in a vicinity of $w^0$. This is accomplished by moving along the steepest descent direction, $d^0\in {\H}$, of $\J$ at $w^0$, given by
\[
(d^0,v)_{\H}:=-\dprod{\J'(w^0),v}\qquad\forall v\in {\H},
\]
where $\dprod{\cdot,\cdot}$ signifies the dual product. Note that this is a \emph{linear} problem. Then, for a suitable step size $s^0>0$, we replace $v^0$ by a new direction
\begin{equation}\label{eq:s0}
v^1:=v(s^0),
\end{equation}
where, for $s>0$, we define
\begin{equation}\label{eq:v(s)}
v(s):=\frac{v^0+sd^0}{\norm{v^0+sd^0}_{\H}}  \in \sL; 
\end{equation}
here, notice that $v^0+sd^0 \in L^\perp$ since $p(v^0)$ is a local minimum of $\E$ on $L_{v^0}$, and $d^0$ is the steepest descent direction of $\E$ at $p(v^0)$ (cf.~\cite[Lem.~1.2]{LiZhou:02}); thus, $v(s) \in \sL$. Furthermore, under the assumptions (a) and (b) in the above Proposition~\ref{pr:LZ}, if $w^0$ is \emph{not} a critical point of $\J$ (i.e. $\J'(w^0)\neq 0$), then, for any $\delta>0$ with 
\begin{equation}\label{eq:deltaE'}
\norm{\J'(w^0)}_{{\H}^\star}>\delta,
\end{equation}
there exists $\sigma_0>0$ such that it holds the descent property
\begin{align} \label{eq:stepsize}
\J(p(v(s)))-\J(w^0)<-\alpha\delta\norm{v(s)-v^0}_{\H}\qquad\forall s\in(0,\sigma_0).
\end{align}
Hence, selecting the step size $s^0\in(0,\sigma_0)$ in~\eqref{eq:s0} appropriately, and letting 
\begin{equation}\label{eq:w1}
w^1:=p(v^1)=u^1+t^1v^1,
\end{equation} 
for some $u^1\in L$ and $t^1\ge 0$, we obtain $\J(w^1)<\J(w^0)$.
\end{enumerate}

Upon repeating step (ii), we obtain a sequence $\{w^k=p(v^k)\}_{k}\subset\imag(p)$ such that $\J(w^k)$ is strictly monotonically decreasing with respect to $k$. The following result, which summarizes \cite[Thm.~3.1 and 3.2]{LiZhou:02}, attends to the convergence of the above process.

\begin{proposition} \label{pr:convergence1}
 Let $p$ be a peak selection of $\J$ (with respect to $L$), and suppose that $\J$ satisfies the Palais--Smale condition~{\rm (PS)}. If 
 \begin{enumerate}[\rm (a)]
  \item $p$ is continuous,
  \item $\inf_{u\in L}\norm{w^k-u}_{\H} \geq \alpha$ for all $k=0,1,2,\dotsc$, for some constant $\alpha>0$, and
  \item $\inf_{v \in \sL} \J(p(v)) > -\infty$, 
 \end{enumerate}
then $\{w^k\}_{k}$ has a converging subsequence. Moreover, any such convergent subsequence tends to a critical point of $\J$. 
\end{proposition}

For the above iterative scheme, this theorem asserts that we can find $k^\star$ large enough  such that $\norm{\J'(w^{k^\star})}_{{\H}^\star}$ is sufficiently small. We then let $w_n:=w^{k^\star}$, and restart the search for a new critical point.

\begin{remark}
It is sensible to choose $v^0\in\sL$ in~\eqref{eq:ascent1} to be an ascent direction of $\J$ at $w_{n-1}$, i.e.
\begin{equation}\label{eq:ascent}
\J(w_{n-1}+tv^0)>\J(w_{n-1}),
\end{equation}
for any $t>0$ sufficiently small. Indeed, moving along an ascent direction works in favour of assumption (b) in Proposition~\ref{pr:convergence1} (see also Theorem~\ref{thm:fullconvergence} below).
\end{remark}

\section{An adaptive local minimax Galerkin method} \label{sc:LMMG}

The purpose of this section is to provide a practical scheme for the approximation of the nonlinear equation~\eqref{eq:F=0}. To this end, we employ a sequence of finite-dimensional Galerkin subspaces $\X_N\subset {\H}$, $N\ge 0$, with the hierarchical property $\X_0\subset \X_1\subset\ldots\subset {\H}$. In order to deal with the nonlinearity of~$\F$, based on a suitable a posteriori error analysis, we introduce an adaptive interplay between the minimax scheme from \S\ref{sc:minimax} and the Galerkin discretizations. We thereby obtain an \emph{iterative minimax Galerkin method}. 

\subsection{Galerkin discretization}

For $N\ge 0$, the Galerkin discretization of~\eqref{eq:J'=0} is to find discrete approximations $w_N\in \X_N$ such that
\[
\R_N(w_N)=0\qquad\text{in }\X_N,
\] 
where, for $w\in {\H}$, we define the discrete residual $\R_N(w)\in \X_N$ by
\begin{equation}\label{eq:RN}
(\R_N(w),v)_{\H}:=\dprod{\J'(w),v}\qquad\forall v\in \X_N.
\end{equation}
Denoting by $\{w_N^k\}_k\subset\X_N$ the sequence generated by the local minimax procedure from \S\ref{sc:minimax} on a given Galerkin space $\X_N$, it holds that
\begin{equation}\label{eq:R0}
\norm{\R_N(w_N^k)}_{\H}
=\sup_{\genfrac{}{}{0pt}{}{v\in \X_N}{\norm{v}_{\H}=1}}\frac{\dprod{\J'(w_N^k),v}}{\norm{v}_{\H}}
=:\norm{\J'(w_N^k)}_{\X_N^\star}
 \to 0,
\end{equation}
for $k\to\infty$; see \cite[Thm.~4.1]{Yao:15}. Hence, if we consider a given positive function $\sigma:\mathbb{N}_0 \to (0,\infty)$ with
\[
 \sigma(N) \to 0 \qquad \text{for} \qquad N \to \infty,
\]
then, for any $N \in \mathbb{N}_0$, there is $k_N\in\mathbb{N}$ such that 
\begin{equation}\label{eq:RN00}
\norm{\R_N(w_N^k)}_{{\H}} < \sigma(N),
\end{equation} 
for all $k\ge k_N$. 

\subsection{Adaptive local minimax Galerkin procedure}

The main idea of the Algorithm~\ref{alg:LMMG} in the current work is to provide an adaptive interplay between the following two strategies:
\begin{enumerate}[(I)]
\item \emph{Mountain pass approximation:} On a given Galerkin space $\X_N$, for $N\ge 0$, we run the local minimax procedure from \S\ref{sc:minimax} on $\X_N$ until the resulting approximations, $w_N^k\in \X_N$, $k\ge 0$, are sufficiently close to a zero of $\R_N$. Let us denote by $w_N^{k^\star}$ the final approximation on the present Galerkin space $\X_N$. From the point of view of computational complexity, we note that the core part of the minimax Galerkin discretization is the solution of the \emph{linear} problem~\eqref{eq:RN}, together with the application of the peak selection in each iteration step.
\item \emph{Adaptive Galerkin discretization:} Once the norm of the residual obtained from step~(I), i.e. $\norm{\R_N(w_N^{k^\star})}_{\H}$, is small enough, we enrich the Galerkin space $\X_N$ appropriately. This is based on the assumption that we have at our disposal a computable \emph{error indicator} $\eta_N:\,\X_N\to[0,\infty)$ such that
\begin{equation}\label{eq:I1}
\norm{\J'(v)}_{{\H}^\star} \leq C \eta_N(v)\qquad\forall v\in \X_N,
\end{equation}
for some constant $C>0$ (independent of $\X_N$); here, the dual norm is defined by
\begin{equation}\label{eq:dualnorm}
\norm{\phi}_{{\H}^\star}:=\sup_{y\in S_{\X}}\dprod{\phi,y}. 
\end{equation}
Furthermore, we suppose that $\eta_N(v)$ comprises of \emph{local} error contributions which allow to refine the Galerkin space $\X_N$ effectively. 
\end{enumerate}

This iterative Galerkin scheme is outlined in Algorithm~\ref{alg:LMMG}.

\begin{algorithm} 
\caption{Adaptive local minimax Galerkin (LMMG) algorithm}
\label{alg:LMMG}
\begin{algorithmic}[1]
\State {Input $n-1$ previously found critical points $w_{1},\dotsc,w_{n-1}$ of $\J$ as in~\eqref{eq:w1n}. }
\State {Prescribe a steering parameter $\gamma>0$, a step size control $\lambda>0$, and a tolerance $\epsilon_{\mathrm{tol}}>0$.}
\State {Start with an initial Galerkin space $\X_0 \subset {\H}$, and set $N=0$.}
\State {Let $L_0:=\mathrm{span}\{w_{{0,1}}, \ldots, w_{0,n-1}\}\subset\X_0$, with suitable approximations (e.g. a nodal interpolant in the context of finite element discretizations) $w_{0,1},\ldots,w_{0,n-1}$ of $w_0,\ldots,w_{n-1}$ in $\X_0$, respectively.}
\State {Choose $v_0^0 \in S_{L_0^\perp} \cap \X_0$ to be an (ascent) direction at $w_{0,n-1}$, cf.~\eqref{eq:ascent}.}
\Repeat
\myState {Set $k=0$.}
\myState{Use a peak selection $p_N$ on $\X_N$, and determine $w_N^0 =p_N(v_N^0)=t_N^0v_N^0+u_N^0$.}
\While {$\norm{\R_N(w_N^k)}_{{\H}} > \gamma \eta_N(w_N^k)$ or $\norm{\R_N(w_N^k)}_{{\H}} > \sigma(N)$}
\parState{Compute the steepest descent direction of $\J$ at the point $w_N^k$, i.e. solve the linear discrete problem~\eqref{eq:RN} to define
\begin{equation}\label{eq:sd}
d_N^k:=-\R_N(w_N^k).
\end{equation}}
\parState {Set
\begin{align}\label{eq:ud}
 v_N^k(s)=\frac{v_N^k+sd_N^k}{\norm{v_N^k+sd_N^k}_{{\H}}}.
\end{align}
}
\parState {Compute $s_N^k$ from~\eqref{eq:sk}.}
\parState {Set $v_N^{k+1}=v_N^k(s_N^k)$, and determine $w_N^{k+1}=p_N(v_N^{k+1})=t_N^{k+1}v_N^{k+1}+u_N^{k+1}$, for unique $u_N^{k+1}\in L_N$ and $t_N^{k+1}\ge 0$, cf.~\eqref{eq:w1}.}
\parState {Update $k\gets k+1$.}
\EndWhile
\myState {Enrich the Galerkin space $\X_N$ appropriately using the local error indicator $\eta_N(w_N^k)$.}
\myState {Define $w_{N+1}^0 := w_N^k$ and $v_{N+1}^0 := v_N^k$ by inclusion $\X_{N+1} \hookleftarrow \X_N$.}
\myState {Update $N \gets N+1$.}
\myState{Let $L_N=\mathrm{span} \{w_{{N,1}}, \dotsc, w_{N,n-1}\} \subset \X_N$, where $w_{N,i}$ signifies a suitable approximation (cf.~line 4) of $w_i$ in $\X_N$, for $1\le i\le n-1$.}
\Until {$\norm{\R_N(w_N^{k})}_{\H} < \epsilon_{\mathrm{tol}}$.}
\end{algorithmic}
\end{algorithm}

\begin{remark}
We comment on two aspects of Algorithm~\ref{alg:LMMG}.
\begin{enumerate}[(a)]
\item The stopping criterion for the iteration is expressed in terms of the inequality
\begin{equation}\label{eq:stop}
\norm{\R_N(w_N^{k^\star})}_{{\H}}\le
\min\left(\gamma\eta_N(w_N^{k^\star}),\sigma(N)\right),
\end{equation}
where $\gamma>0$ is a prescribed method parameter. In contrast to the expression $\norm{\E'(w_N^{k^\star})}_{\X^\star}$, cf.~\eqref{eq:I1}, the quantity $\norm{\R_N(w_N^{k^\star})}_{{\H}}$ is computable in practice. The second condition from the stopping criterion~\eqref{eq:stop}, i.e. $\norm{\R_N(w_N^k)}_{{\H}} \leq \sigma(N)$, was introduced for the rather theoretical purpose of Theorem~\ref{thm:fullconvergence} below. Indeed, running the experiments in \S\ref{sec:ex} without this stopping criterion seems to yield the same solutions, with the same asymptotic convergence behaviour and comparable number of iteration steps. Yet, if $\sigma(N)$ decays too fast for increasing $N$, then unnecessary iteration steps are performed (which, desirably, should be prevented). 
\item Following~\cite[p.~870]{LiZhou:02}, the step size for the update of the ascent direction from~\eqref{eq:v(s)}, applied on the Galerkin space $\X_N$, is defined by
\begin{equation}\label{eq:sk}
s^k_N=\frac{\lambda}{2^{m}},
\end{equation}
where $\lambda>0$ is a step size control parameter, and $m\in\mathbb{Z}$ is the minimal integer such that
\[
2^m > \norm{d_N^k}_{{\H}},
\]
and
\begin{equation}\label{eq:Jdec}
\J\left(p\left(v_N^k\left(\nicefrac{\lambda}{2^m}\right)\right)\right)-\J(w_N^k) \leq -\frac12t_N^k\norm{d_N^k}_{\X}\norm{v_N^k\left(\nicefrac{\lambda}{2^m}\right)-v_N^k}_{{\H}}.
\end{equation}
Here, $d^k_N$ and $v^k_N(s)$ are the steepest descent direction and updated ascent direction from~\eqref{eq:sd} and~\eqref{eq:ud}, respectively. We observe that the estimate~\eqref{eq:Jdec} follows from~\eqref{eq:stepsize} upon letting
\begin{align*}
\alpha&:=\frac{1}{\sqrt{2}}t_N^k<t_N^k=\inf_{u\in L}\norm{w_N^k-u}_{\H},
\intertext{and}
\delta&:=\frac{1}{\sqrt{2}} \norm{d_N^k}_\X < \norm{d_N^k}_\X \leq \norm{\E'(w_N^k)}_{\X^\star},
\end{align*}
cf.~Proposition~\ref{pr:LZ}~(b) and~\eqref{eq:deltaE'}, respectively, whenever $m$ is chosen large enough.
\end{enumerate}
\end{remark}

\begin{theorem} \label{thm:fullconvergence}
Suppose that the functional $\J$ satisfies the Palais--Smale condition~{\rm (PS)} on $\X$, and let $\{\X_N\}_{N\ge0}$ be the sequence of finite dimensional subspaces of $\X$ generated by running Algorithm~\ref{alg:LMMG} (for $\epsilon_{\mathrm{tol}}=0$). Furthermore, for every $N\ge0$, consider the closed subspace $L_N \subset \X_N$ from line 19, and assume that the peak selection $p_N$ of $\J$ with respect to $L_N$ fulfils the following properties:
 \begin{enumerate}[\rm (a)]
  \item $p_N$ is continuous,
  \item $\inf_{y\in L_N}\norm{w_N^k-y}_{\H} \geq \alpha_N >0$ for all $k\ge0$, and
  \item $\inf_{N \geq 0}\inf_{v \in S_{L_N^\perp} \cap \X_N} \J(p_N(v)) > -\infty$.
 \end{enumerate}
Then, the following two convergence properties hold true:
\begin{enumerate}[\rm(I)]
 \item If there exists an index $N \geq 0$ such that the while loop (line 9--line 15) in Algorithm~\ref{alg:LMMG} does not terminate, and if $\E':\,\X\to\X^\star$ is continuous on the subspace $\X_N$, then the generated sequence $\{w_N^k\}_{k \geq 0}\subset\X_N$ converges to the set of critical points of $\J$, given by 
 \begin{equation}\label{eq:CJ}
C_{\J}:=\{w \in {\H}: \J'(w)=0 \ \text{in } {\H}^\star\},
\end{equation} 
in the sense that $\lim_{k\to\infty}\inf_{y\in C_{\J}}\norm{w_N^{k}-y}_{\H} =0$.
 \item Otherwise, for each $N\ge 0$, suppose that there exists a finite integer $k^\star=k^\star(N)\in\mathbb{N}$ such that~\eqref{eq:stop} is satisfied. Moreover, let the sequence of Galerkin spaces $\{\X_N\}_{N\ge0}$ be dense in $\X$, i.e. $\X=\overline{\bigcup_{N \geq 0} \X_N}$. Furthermore, for any infinite subset $J\subseteq \mathbb{N}$, assume that the corresponding sequence $\{\E'(w_{N}^{k^\star})\}_{N\in J}$ has a converging subsequence in $\X$.
Then, we have
$
\inf_{y\in C_{\J}}\norm{w_N^{k^\star}-y}_{\H} \to 0,
$
as $N\to\infty$, with $C_{\J}$ from~\eqref{eq:CJ}.
\end{enumerate}
\end{theorem}

\begin{proof}[Proof of {\rm(I)}]
If, for some $N\ge0$, the while loop does not terminate in finitely many steps, i.e. \eqref{eq:stop} is never satisfied, then, taking into account~\eqref{eq:RN00}, we conclude that
$
\gamma\eta_N(w_N^k)\le\norm{\R_N(w_N^k)}_{{\H}}
$
for $k$ sufficiently large. Hence, due to~\eqref{eq:I1} and~\eqref{eq:R0}, it follows that
\[
\norm{\J'(w_N^k)}_{{\H}^\star} \leq C \eta_N(w_N^k) \leq \frac{C}{\gamma} \norm{\R_N(w_N^k)}_{{\H}} \to 0 \quad \text{for} \quad k \to \infty.
\] 
We will now show that this, in turn, yields 
\[d^k_N:=\inf_{y\in C_{\J}}\norm{w_N^{k}-y}_{\H} \to 0,\] for $k \to \infty$. Assume this were false. Then, there would exist some $\delta >0$ and a subsequence $\{w_{N}^{k_j}\}_j$ with $d^{k_j}_{N} \geq \delta$ for all $k_j \geq 0$. Owing to assumption (c) and the monotonicity property~\eqref{eq:Jdec}, we notice that the sequence $\{\J(w_{N}^{k_j})\}_j$ is bounded. Furthermore, we have $\J'(w_{N}^{k_j}) \to 0$ in $\X^\star$ for $j \to \infty$. Then, exploiting the Palais--Smale compactness condition~(PS) for $\E$, there is a convergent subsequence $w_{N}^{k_{j_i}}$, with a limit $w_N^\star\in\X_N$. By continuity of $\E'$ on $\X_N$, it holds that $\J'(w_N^\star)=0$. This means that $w_N^\star$ is a critical point of $\J$, i.e. $d^{k_{j_i}}_N\to0$ for $j\to\infty$, which constitutes a contradiction.
\end{proof}

\begin{proof}[Proof of {\rm(II)}]
For any $N\ge 0$, define $\g_N:=\J'(w_N^{k^\star}) \in {\H}^\star$, where $w^{k^\star}_N$ is the final approximation on the Galerkin space $\X_N$. We aim to show that $\g_N \to 0$ in ${\H}^\star$ for $N \to \infty$; the remainder of the proof then follows as in ~(I). Suppose to the contrary that there exists $\delta>0$ and an infinite subset $J\subset\mathbb{N}$ such that $\norm{\g_{N}}_{{\H}^\star} \geq \delta$ for all $N\in J$. By assumption,
there exists a subsequence $\{\g_{N_{\ell}}\}_{\ell}$, with $\{N_\ell\}_\ell\subset J$, which converges to some limit $\g_\infty \in \X^\star$:
\begin{equation}\label{eq:aux20200902}
\norm{\g_\infty-\g_{N_{\ell}}}_{\X^\star} \to 0\quad\text{ for }\quad\ell \to \infty.
\end{equation}
Set $\gamma:=\norm{\g_\infty}_{\X^\star}\ge\delta$, and
let $\varepsilon:=\nicefrac{1}{4}\min\{\delta,\delta\gamma^{-1}\}\le \nicefrac14$. Pick any $x \in \X$ with $\norm{x}_\X=1$.
By density, there exists $x_\varepsilon \in \bigcup_{N \geq 0} \X_N$ such that $\norm{x-x_\varepsilon}_\X \leq \varepsilon$, and, in particular, $x_\varepsilon \in \X_{N}$ for any $N$ large enough. Therefore, applying the triangle inequality, and exploiting the linearity of the involved operators, we find
\begin{equation*} 
\begin{split}
|\g_\infty(x)| &\leq |\g_\infty(x-x_\varepsilon)|+|\g_\infty(x_\varepsilon)-\g_{N_{\ell}}(x_\varepsilon)|+|\g_{N_{\ell}}(x_\varepsilon)|   \\
& \leq \varepsilon\norm{\g_\infty}_{\X^\star}  + (1+\varepsilon)\Big(\norm{\g_\infty-\g_{N_{\ell}}}_{\X^\star}+\norm{\g_{N_{\ell}}}_{\X_{N_{\ell}}^\star}\Big),
\end{split}
\end{equation*}
for $\ell$ large enough. For the last term, recalling~\eqref{eq:R0}, we obtain
\[
\norm{\g_{N_{\ell}}}_{\X_{N_{\ell}}^\star}=\norm{\R_{N_{\ell}}(w_{N_{\ell}}^{k^\star})}_{\H}
\leq \sigma(N_{\ell}) \to 0 \quad \text{for} \quad \ell \to \infty.
\]
Therefore, invoking~\eqref{eq:aux20200902}, for sufficient large $\ell$, it holds that 
$
\norm{\g_\infty-\g_{N_{\ell}}}_{\X^\star}+\norm{\g_{N_{\ell}}}_{\X_{N_{\ell}}^\star} \leq \nicefrac{\varepsilon}{2}.
$ 
In summary, by our choice of $\varepsilon$, 
we infer that $|\g_\infty(x)| \leq \gamma \varepsilon + \varepsilon \le \nicefrac{\delta}{2}$. Thus, since $x \in \X$ was arbitrary, we deduce that $\norm{\g_\infty}_{\X^\star} < \delta$ which yields a contradiction.
\end{proof}

\begin{remark} \label{rem:convergence}
We add a few comments on part~{\rm(II)} of Theorem~\ref{thm:fullconvergence} above.
\begin{enumerate}[(i)]
\item From the proof of (II) we deduce that $\J'(w^{k^\star}_N)\to0$ in $\X^\star$ for $N\to\infty$. Hence, from boundedness of $\{\E(w^{k^\star}_N)\}_{N\ge0}$, cf.~part~(I), and owing to the Palais-Smale condition~{\rm(PS)}, we infer that the sequence $\{w_N^{k^\star}\}_{N\ge0}$ has a convergent subsequence in $\X$. In addition, any convergent subsequence tends to some critical point in $C_\E$.
\item If $\E':\,\X \to \X^\star$ is a (nonlinear) compact mapping, and the sequence $\{w_{N}^{k^\star}\}_{N\ge0}$ is bounded in $\X^\star$, then any subsequence of $\{\E'(w_{N}^{k^\star})\}_{N\ge0}$ possesses a converging subsequence, as required in part~{\rm(II)}.
\item If the sequence $\{\E'(w_N^{k^\star})\}_N \subset \X^\star$ is merely bounded, without presuming the existence of converging subsequences, it still holds that $\{\E'(w_N^{k^\star})\}_N$ converges weakly to 0. This can be derived from the proof of Theorem~\ref{thm:fullconvergence} upon replacing the strong topology by the weak topology. Moreover, if we make the uniform approximation assumption that
\[
 \sup_{u \in C_{\E}}\inf_{v_N \in \X_N} \norm{u-v_N}_{\H} \to 0 \qquad \text{as} \qquad N \to \infty,
\]
then it follows 
\begin{equation*}
d_N:=\inf_{y\in C_{\J}}\norm{w_N^{k^\star}-y}_{\H} \to 0 \qquad \text{as} \qquad N \to \infty,
\end{equation*}
where $C_\E$ denotes the set of critical points of $\E$ from~\eqref{eq:CJ}.
\item In the context of finite element methods, the space $\bigcup_{N\ge0}\X_N$ is typically dense in $\X$ whenever the mesh size (i.e. the maximal element diameter) tends to zero; cf., e.g., \cite[Thm.~3.2.3]{CiarletBOOK}.
\end{enumerate}
\end{remark}

\begin{cor}
Given the assumptions of Theorem~\ref{thm:fullconvergence}~{\rm(II)}. If there exists a constant $\alpha>0$ such that $\alpha_N \geq \alpha>0$ uniformly for all $N \geq 0$, and $w_{N,i} \to w_i$ for $N \to \infty$, $i\in\{1,\dots,n-1\}$, then any limit point $w^\star$, as mentioned in Remark~\ref{rem:convergence}~{\rm(i)}, is a \emph{new} critical point of $\E$, i.e. $w^\star \neq w_i$ for each $i \in \{1,\dotsc,n-1\}$.
\end{cor}

\begin{proof}
Suppose to the contrary that there exists a subsequence $\{w_{N_j}^{k^\star}\}_j$ with limit $w^\star=w_i$ for some $i \in \{1,\dotsc,n-1\}$. Then, for $N_j \geq 0$ large enough, it holds that 
\[
\norm{w_{N_j}^{k^\star}-w_i}_\X=\norm{w_{N_j}^{k^\star}-w^\star}_\X \leq \nicefrac{\alpha}{4}\quad\text{ and }\quad
\norm{w_i-w_{N_j,i}}_X \leq \nicefrac{\alpha}{4}.
\] 
Hence, applying the triangle inequality leads to
\[\norm{w_{N_j}^{k^\star}-w_{N_j,i}}_\X \leq \norm{w_{N_j}^{k^\star}-w_i}_\X + \norm{w_i-w_{N_j,i}}_X \leq \nicefrac{\alpha}{2}.
\]
Recalling that $\alpha_N\ge\alpha>0$ uniformly, this contradicts assumption (b) of Theorem~\ref{thm:fullconvergence}.
\end{proof}

\section{Application to semilinear elliptic PDE}\label{sc:SL}

We apply the adaptive LMMG Algorithm~\ref{alg:LMMG} in the context of finite element discretizations of semilinear elliptic Dirichlet boundary value problems of the form
\begin{subequations}\label{eq:pde}
\begin{alignat}{2}
 - \varepsilon \Delta u + qu &=f(\cdot,u)& \quad& \text{in } \Omega \label{eq:modelproblem}\\
 u&=0 &&\text{on } \partial \Omega; \label{eq:modelproblemboundary}
\end{alignat}
\end{subequations}
here, we assume that $\Omega \subset \mathbb{R}^n$ is a (Lebesgue-measurable) bounded open domain, and $\varepsilon>0$ is a singular perturbation parameter. Furthermore, $q \in \mathrm{L}^\infty(\Omega)$ satisfies the following condition: there are two constants $\nu \geq 0$ and $c_\nu \geq 0$, which do not depend on $\varepsilon$, such that 
\begin{align} \label{eq:qbound}
q \geq \nu \quad \text{in } \Omega \qquad \text{and} \qquad \norm{q}_{\L^\infty(\Omega)} \leq c_\nu \nu;
\end{align}
cp.~\cite[\S4.4, (A3)]{Verfurth:13}. Moreover, the right-hand side function $f$ satisfies the following standard conditions (see, e.g., \cite[p.~9]{Rabinowitz:86}):
\begin{itemize} 
 \item[(f1)] $f \in \C(\overline{\Omega} \times \mathbb{R};\mathbb{R})$;
 \item[(f2)] there are constants $a_1,a_2\ge 0$ such that
\[
  |f(x,t)| \leq a_1 + a_2 |t|^s,
\]
where $0 \leq s < \nicefrac{(n+2)}{(n-2)}$ if $n > 2$, and 
\[
 |f(x,t)| \leq a_1 \exp(\varphi(t)),
\]
where $\varphi(t)t^{-2} \to 0$ as $|t| \to \infty$, if $n=2$. 
\end{itemize}
In the one-dimensional case, $n=1$, we note that (f2) can be dropped.

\subsection{Existence of weak solutions}

Under the above conditions it can be shown that critical points of the functional $\E_\varepsilon:\,\H \to \mathbb{R}$, defined by
\begin{align} \label{eq:energyfunctional}
 \E_\varepsilon(u):= \frac{\varepsilon}{2} \int_\Omega |\nabla u(x)|^2 \dx+\frac{1}{2} \int_\Omega q(x) u(x)^2 \dx -\int_\Omega F(x,u(x)) \dx,
\end{align}
with 
\begin{equation}\label{eq:F}
F(x,t):=\int_0^t f(x,s) \ds
\end{equation} 
the anti-derivative of $f$, are weak solutions of \eqref{eq:modelproblem}--\eqref{eq:modelproblemboundary}
in the standard Sobolev space $\H:=\HS(\Omega)$; see, e.g., \cite[Prop.~B.10.]{Rabinowitz:86}. Moreover, the functional $\E_\varepsilon \in \C^1(\HS(\Omega);\mathbb{R})$ is well-defined, and an elementary calculation reveals that 
\begin{align} \label{eq:Jprime}
 \dprod{\E_\varepsilon'(u),v}=\varepsilon \int_\Omega \nabla u \cdot \nabla v \dx+\int_\Omega quv \dx -\int_\Omega f(x,u)v \dx \qquad \forall v \in \HS(\Omega).
\end{align}
For our purpose, we endow the space $\HS(\Omega)$ with the inner product defined by 
\begin{align} \label{eq:epsilonip}
 (u,v)_\varepsilon:=\varepsilon\int_\Omega  \nabla u \cdot \nabla v \dx+ \nu\int_\Omega uv \dx,
\end{align}
and the induced norm 
\begin{align}\label{eq:epsilonnorm}
 \norm{u}^2_{\X}:=\norm{u}_\varepsilon^2:= \varepsilon\int_\Omega  |\nabla u|^2 \dx+\nu \int_\Omega u^2 \dx,
\end{align}
where $\nu$ is the constant from~\eqref{eq:qbound}. We note that this norm is equivalent to the standard $\HS(\Omega)$-norm (with equivalence constants depending on~$\varepsilon$ and~$\nu$); in particular the space $\HS(\Omega)$ equipped with the above inner product is a Hilbert space. 

If we state some additional conditions on $f$, then the functional $\Ee$ from \eqref{eq:energyfunctional} satisfies the Palais--Smale compactness condition (PS) on $\HS(\Omega)$. Specifically, we assume that 
\begin{enumerate}
 \item[(f3)] $f(x,t)=o(|t|)$ as $t \to 0$, and
 \item[(f4)] there are constants $\mu > 2$ and $r \geq 0$ such that
 \[0 < \mu F(x,t) \leq t f(x,t) \qquad \forall |t| \geq r.\]
\end{enumerate}
We remark that integrating (f4) yields the existence of constants $a_3,a_4 >0$ such that
\begin{align} \label{eq:f4cons}
 F(x,t) \geq a_3|t|^\mu-a_4 \qquad \forall x\in\overline\Omega\quad\forall t \in \mathbb{R};
\end{align}
cf.~\cite[Rem.~2.13]{Rabinowitz:86}. If $f$ satisfies (f1)--(f4), then $\Ee$ from \eqref{eq:energyfunctional} does indeed fulfil the PS-condition; we refer to~\cite[p.~11]{Rabinowitz:86} for a detailed analysis. Moreover, invoking the mountain pass theorem, Theorem~\ref{thm:mountainpass}, these assumptions yield the existence of a nontrivial weak solution to \eqref{eq:modelproblem}--\eqref{eq:modelproblemboundary}; see~\cite[Thm.~2.15]{Rabinowitz:86}. Furthermore, we may obtain a nontrivial classical solution, provided that the domain $\Omega$ is sufficiently smooth, and (f1) is replaced by the following stronger condition:
\begin{enumerate}
 \item[(f1')] $f$ is locally Lipschitz continuous in $\overline{\Omega} \times \mathbb{R}$.
\end{enumerate}
Within this setting, i.e. for a sufficiently smooth domain $\Omega$, and assuming (f1') and (f2), it follows that any weak solution of \eqref{eq:modelproblem}--\eqref{eq:modelproblemboundary} is in fact classical; see, e.g. \cite{Agmon:59}. Moreover, supposing that (f1') and (f2)--(f4) hold true, then there exists a positive and negative classical solution; we refer to \cite[Cor.~2.23]{Rabinowitz:86}.

\subsection{Galerkin discretization}

We consider a sequence of hierarchically enriched conforming finite-dimensional subspaces $\X_0\subset \X_1\subset\ldots\subset\HS(\Omega)$. In the specific context of the semilinear boundary value problem~\eqref{eq:pde} the steepest descent direction from~\eqref{eq:sd} with respect to the inner product defined in~\eqref{eq:epsilonip} satisfies the following \emph{linear Galerkin formulation}: Given $w^k_N\in\X_N$, find $d^k_N\in\X_N$ such that
\begin{align} \label{eq:steepestdescent}
\varepsilon\int_\Omega \nabla d_N^k \cdot \nabla v \dx+ \nu \int_\Omega d_N^k v \dx=-\varepsilon\int_\Omega \nabla w_N^k \cdot \nabla v \dx - \int_\Omega q w_N^k v \dx+\int_\Omega f(x,w_N^k)v \dx,
\end{align}
for all $v\in\X_N$. This is the key part in the LMMG Algorithm~\ref{alg:LMMG} (in addition to the optimization process associated to the peak selection). It underlines that the discrete solution of~\eqref{eq:pde} splits into a sequence of linear discrete problems which is obtained iteratively on each of the Galerkin spaces $\{\X_N\}_{N}$. Incidentally, from \eqref{eq:steepestdescent} we immediately see that $d_N^k = 0$ if and only if
\begin{align*}
\int_\Omega \varepsilon \nabla w_N^k \cdot \nabla v \dx+\int_\Omega q w_N^kv \dx= \int_\Omega f(x,w_N^k)v \dx \qquad \forall v \in {\X_N},
\end{align*}
i.e. if and only if $w_N^k$ is the Galerkin solution of~\eqref{eq:pde} in $\X_N$.

\subsection{Convergence of the adaptive LMMG algorithm in the context of semilinear elliptic PDE} 

We aim to solve the semilinear elliptic problem \eqref{eq:modelproblem}--\eqref{eq:modelproblemboundary} by applying the adaptive LMMG Algorithm~\ref{alg:LMMG} to the functional $\Ee$ from \eqref{eq:energyfunctional}. In order for the assumptions of Theorem~\ref{thm:fullconvergence} to hold, we introduce two additional properties of $f$:
\begin{enumerate}
 \item[(f5)] For any given $x\in\Omega$, the function
 \[
 t\mapsto\begin{cases}
 0&\text{if }t=0,\\ 
 |t|^{-1}f(x,t) &\text{if }t\neq 0,
 \end{cases}
 \]
 is strictly increasing in $t$,
 \item[(f6)] For any $x\in\Omega$, the function $f(x,t)$ is continuously differentiable with respect to $t$.
\end{enumerate}
We focus on the case $L=\{0\}$. Then, $L^\perp=\X=\HS(\Omega)$ and $\sL=S_{\H}=\{v \in \HS(\Omega): \norm{v}_{\varepsilon}=1\}$. Recall that the energy functional $\E_\varepsilon$ satisfies the Palais--Smale condition (PS) provided that $f$ features the properties (f1)--(f4). We will now verify the conditions (a)--(c) in Proposition~\ref{pr:convergence1}, as well as the boundedness of $p(\cdot)$ and $\E'(p(\cdot))$, cf.~Theorem~\ref{thm:fullconvergence} and Remark~\ref{rem:convergence}. The following proposition is a summary of the results in \cite[\S 4]{LiZhou:01}; its proof is presented in Appendix~\ref{sc:app}.
 
\begin{proposition}\label{pr:app}
Let $p$ be a peak selection of $\Ee$ from \eqref{eq:energyfunctional} with respect to $L:=\{0\}$. If $f$ satisfies {\rm (f1)--(f6)}, then 
  \begin{enumerate}[\rm (i)]
   \item the peak selection $p$ is uniquely defined and continuous,
   \item $\inf_{v \in S_{\H}} \Ee(p(v)) \geq \rho$ for some $\rho>0$,
   \item $\inf_{y\in L}\norm{p(v)-y}_{\varepsilon}=\norm{p(v)}_{\varepsilon} \geq \alpha$ for some $\alpha >0$ and for all $v \in S_{\H}$, 
   \item $\norm{p(v)}_{\varepsilon} \leq \beta$ and $\norm{\Ee'(p(v))}_{\H^\star} \leq \gamma$ for some $\beta,\gamma >0$ and for all $v \in S_{\H}$.
   \end{enumerate}
In particular, the assumptions of Proposition~\ref{pr:convergence1} are satisfied.   
\end{proposition}

\begin{remark}
In the present setting it is not evident whether $\{\E_\varepsilon'(w_{N}^{k^\star})\}_{N\in J}$, for any infinite subset $J \subseteq \mathbb{N}$, has a converging subsequence, i.e. Theorem~\ref{thm:fullconvergence} {\rm (II)} cannot be applied straightaway. Moreover, with regards to Remark~\ref{rem:convergence} {\rm(ii)}, we note that the operator $\E_\varepsilon':\HS(\Omega) \to \mathrm{H}^{-1}(\Omega)$ is not compact. By assumption {\rm(iv)} above, however, Remark~\ref{rem:convergence} {\rm(iii)} implies the weak convergence of $\{\Ee(w_N^{k^\star})\}_N$ to 0. Finally, in certain cases, we notice that a small enough \emph{a posteriori} bound for $\norm{\Ee(w_N^{k^\star})}_{\X^\star}$, cp.~\eqref{eq:resbound} below, may guarantee the existence of a weak solution of~\eqref{eq:pde} in a neighbourhood of $w_N^{k^\star}$; we refer the interested reader to~\cite{Plum:2001,BreuerMcKennaPlum:2003}.
\end{remark}
 
\begin{remark}
For the special case $L=\{0\}$ as in the above Proposition~\ref{pr:app}, the application of the peak selection, cf.~lines 8 and 13 of the LMMG Algorithm~\ref{alg:LMMG}, amounts to a one-dimensional optimization problem. More precisely, we need to minimize the mapping $t \mapsto -\Ee(tv)$ on $\mathbb{R}^+$, for $v=v^{k+1}_N$. Applying differentiation, this can be expressed in terms of a scalar nonlinear equation, viz.
\[
t\left(\varepsilon  \int_\Omega  | \nabla v|^2 \dx+ \int_\Omega q v^2 \dx\right)=\int_\Omega f(x,tv)v \dx,
\]
cf.~\eqref{eq:gderivative} in Appendix~\ref{sc:app}. This yields a unique minimizer $t_v^\star>0$, and, thereby, the evaluation of the peak selection $p(v)=t_v^\star v$. Equivalently, in the singularly perturbed case $0<\varepsilon\ll 1$, a numerically more stable approach is to first compute the unique minimizer $\widehat{t}^\star_v >0 $ of the scaled mapping $t \mapsto -\varepsilon^{-1}\Ee(\varepsilon^{\nicefrac12}tv)$ on $\mathbb{R}^+$, and then to determine $p(v)=\varepsilon^{\nicefrac{1}{2}} \widehat{t}^\star_v v$.  
\end{remark}

\begin{remark}
Once Algorithm~\ref{alg:LMMG}, initiated from $L=\{0\}$, yields an adequate approximate solution $w_N^{k^\star}$, we may restart the procedure with $w_1:=w_N^{k^\star}$ and  $L:=\{w_1\}$; see e.g.~Experiment~\ref{sec:exphenon} (Case 2). The peak selection now amounts to a two-dimensional (constraint) local minimization problem, for which a variety of solvers are available including gradient methods or interior-point schemes. Here, we note that Proposition~\ref{pr:app} does no longer apply for $\dim(L)>0$. Nonetheless, as can be seen from Figure~\ref{fig:Henonq12}, the a posteriori error bound for $\norm{\Ee(w_N^{k^\star})}_{\X^\star}$ may still decay; in combination with the existence and enclosure results from~\cite{Plum:2001,BreuerMcKennaPlum:2003}, this could potentially yield a new solution of~\eqref{eq:pde} in a neighbourhood of $w_N^{k^\star}$.  
\end{remark}

\section{Numerical experiments} \label{sec:numexp}

The aim of this section is to test our adaptive LMMG Algorithm~\ref{alg:LMMG} in the context of the singularly perturbed semilinear elliptic boundary value problem~\eqref{eq:pde}. This requires to solve~\eqref{eq:steepestdescent} on a suitable family of Galerkin spaces. In the sequel, standard low-order finite element discretizations will be applied.

\subsection{Finite element discretization}

We consider regular and shape-regular meshes $\mathcal{T}_N$ that partition the polygonal domain~$\Omega$ into open and disjoint triangles~$T \in\mathcal{T}_N$ such that $\overline{\Omega}=\bigcup_{T \in \mathcal{T}_N} \overline T$. For a triangle $T\in\mathcal{T}_N$, we denote by $h_T$ the diameter of ~$T$. Moreover, we consider the finite element space 
\[
\X_N:=\left\{v \in \HS(\Omega):\, v|_T \in \P_p(T) \ \forall T \in \mathcal{T}_N\right\},
\]
where, for fixed~$p\in\mathbb{N}$, we signify by $\P_p(T)$ the space of all polynomials of total degree at most $p \geq 1$ on $T \in\mathcal{T}_N$. In particular, in our numerical experiments below, we set $p=1$.

Within the adaptive LMMG framework, we will consider a sequence of finite element meshes $\{\mathcal{T}_N\}_{N}$, whereby we start with an initial (coarse) triangulation~$\mathcal{T}_0$ of $\Omega$. All subsequent meshes are obtained by (regular) refinement, i.e. for~$N\ge 0$, the mesh $\mathcal{T}_{N+1}$ is a hierarchical refinement of $\mathcal{T}_N$. 

For an edge $e \subset \partial T^+ \cap \partial T^{-}$, which is the intersection of two neighbouring elements $T^{\pm} \in \mathcal{T}_N$, we signify by $\jmp{\vec{v}}|_e=\vec{v}^{+}|_e \cdot \vec{n}_{T^+}+\vec{v}^{-}|_e\cdot \vec{n}_{T^{-}}$ the jump of a (vector-valued) function $\vec{v}$ along~$e$, where $\vec{v}^{\pm}|_e$ denote the traces of the function $\vec{v}$ on the edge $e$ taken from the interior of $T^{\pm}$, respectively, and $\vec{n}_{T^{\pm}}$ are the unit outward normal vectors on $\partial T^{\pm}$, respectively.\\

\subsection{Error indicator}

We will use a local error indicator which satisfies~\eqref{eq:I1}. To this end, we pursue a residual-based a posteriori error analysis, which is robust with respect to possibly small values of the singular perturbation parameter~$\varepsilon>0$ appearing in~\eqref{eq:modelproblem}; see, e.g., \cite{Verfurth:13}. Specifically, for a fixed triangulation $\mathcal{T}_N$, and for any finite element approximation $u\in\X_N$, it holds the \emph{$\varepsilon$-robust upper a posteriori residual bound}
\begin{align} \label{eq:resbound}
 \norm{\Je'(u)}_{{\H}^\star} \leq C_I \left(\sum_{T \in \mathcal{T}_N} \eta_T^2(u)\right)^{\nicefrac{1}{2}},
\end{align}
where $\X^\star=\mathrm{H}^{-1}(\Omega)$ signifies the dual space of~$\X=\HS(\Omega)$, equipped with the norm $\norm{\cdot}_{\X^\star}$ from~\eqref{eq:dualnorm}. Furthermore, $C_I>0$ is an interpolation constant (only depending on the polynomial degree $p$ and on the shape-regularity of the mesh, however, \emph{independent} of the singular perturbation parameter~$\varepsilon$, and
\begin{equation} \label{eq:etaloc}
\begin{split}
 \eta_T^2(u)&=\alpha_T^2 \norm{\varepsilon \Delta u - qu+ f_T(\cdot,u)}_{\L^2(T)}^2 \\ 
 &\quad+\frac{1}{2} \varepsilon^{-\nicefrac12}\alpha_T \norm{ \jmp{\varepsilon \nabla u}}_{\L^2(\partial T \setminus \partial\Omega)}^2 +\alpha_T^2 \norm{f(\cdot,u)-f_T(\cdot,u)}_{\L^2(T)}^2,
\end{split}
\end{equation}
where $f_T(\cdot,u)$ denotes a piecewise polynomial approximation of $f(\cdot,u)$ in the (local) finite element space; alternatively, if no data oscillation needs to be considered, it is sufficient to replace $f_T(\cdot,u)$ by $f(\cdot,u)$. Furthermore, for $T \in \mathcal{T}_N$, we let
\[
\alpha_T=\begin{cases}
\min(\nu^{-\nicefrac{1}{2}},\varepsilon^{-\nicefrac12}h_T),&\text{if }\nu\neq 0,\\
\varepsilon^{-\nicefrac12}h_T&\text{if }\nu=0.
\end{cases}
\] 
We refer to~\cite[\S4.4 \& \S5.2]{Verfurth:13} for details, or to~\cite[\S4]{AmreinWihler:15} for an analogous analysis. In the context of piecewise linear approximation, for instance, we remark that the data oscillation term in~\eqref{eq:etaloc} can be bounded from above in the form
\[
\norm{f(\cdot,u)-f_T(\cdot,u)}^2_{\L^2(T)} \leq C h_T \Norm{\frac{\partial f(\cdot,u)}{\partial u}}^2_{\L^\infty(T)} \norm{\nabla u}^2_{\L^2(T)},
\]
where $C$ is a constant depending on the shape-regularity of the mesh; indeed, this results from the approximation property of the Cl\'{e}ment operator (see \cite[Remark 5.17]{Verfurth:13} for a related issue). Finally, we define, for any finite element approximation $u$, 
\begin{align*}
 \eta_N(u):=\left(\sum_{T \in \mathcal{T}_N} \eta_T^2(u)\right)^{\nicefrac{1}{2}},
\end{align*}
which satisfies~\eqref{eq:I1}.

\subsection{Examples}\label{sec:ex}

In our test problems we consider either of the two square domains $\Omega_1:=(0,1)^2 \subset \mathbb{R}^2$ or $\Omega_2:=(-1,1)^2 \subset \mathbb{R}^2$, with the Euclidean coordinates denoted by $\mathbf{x}=(x,y) \in \mathbb{R}^2$. The nonlinearity $f$ in~\eqref{eq:modelproblem} is chosen in compliance with the assumptions (f1)--(f6) from Proposition~\ref{pr:app} for all our experiments. We use an initial mesh consisting of 32 uniform triangles, and run the LMMG algorithm until the number of elements in the mesh exceeds $10^6$. The marking and refinement of elements are based on the D\"orfler strategy~\cite{Doerfler:96} (with parameter $\theta=0.5$), and on the newest vertex bisection method~\cite{Mitchell:91}, respectively. In our performance plots, the residual indicator $\eta_N(u_N^{k^\star})$ and the norm $\norm{\R_N^{k^\star}}_{\varepsilon}$, for $N \geq 1$, will be displayed each time before a mesh refinement is undertaken. In accordance with the expected optimal convergence rate of the $\mathbb{P}_1$-FEM, the function $\sigma:\mathbb{N} \to \mathbb{R}^+$ applied in Algorithm~\ref{alg:LMMG} is defined by $\sigma(N):=\nicefrac{\norm{\R_0^1}_{\varepsilon}}{|\mathcal{T}_N|^{\nicefrac{1}{2}}}$, for $N \geq 1$; in particular, on the initial space $\X_0$, we apply one minimax step only (thereby yielding $\R_0^1$), and then run the while loop in Algorithm~\ref{alg:LMMG} from $N \geq 1$ onwards. 

Our numerical tests indicate that $\gamma=0.25$ and $\lambda=0.5$ are sensible choices for the steering parameter and step size control, respectively. For this setup our computations show that only a few steps (namely, one or two) of the minimax iteration are usually required on each Galerkin space, meaning that no unnecessary iterative steps are performed. In some experiments the number of minimax steps are slightly larger during the initial phase. In addition, we have observed some cases where the number of minimax steps suddenly increases on a specific Galerkin space, together with a rapid decay of the value of the functional~$\E$; this may indicate that the algorithm either follows a more effective steepest descend direction, or it has detected a new critical point.

\subsubsection{Lane-Emden equation} \label{sec:explaneemden}

In our first experiment, which is borrowed from~\cite[\S6, (1)]{choimckenna:93}, we consider the boundary value problem
\[
\begin{alignedat}{2}
 -\Delta u&=u^3 & \quad & \text{in } \Omega_1 \\
u&=0 && \text{on } \partial \Omega_1.
\end{alignedat}
\]
In line with Proposition~\ref{pr:app}, we let $L:=\{0\}$. Furthermore, we select the initial guess~$v_0^0$ to be the linear interpolant of the function $(x,y)\mapsto \sin(\pi x) \sin(\pi y)$ in the element nodes, scaled to unit length in $\X$. A visual comparison of Figure~\ref{fig:LaneEmden} (left) and~\cite[Fig.~8]{choimckenna:93} shows that the same solutions are obtained in both computations. Moreover,  Figure~\ref{fig:LaneEmden} (right) underlines that the residual indicator  decays at an optimal convergence rate of $\mathcal{O}(|\mathcal{T}_N|^{-\nicefrac{1}{2}})$ (indicated by the dashed line).

\begin{figure}[ht] 
 \hfill{\includegraphics[width=0.48\textwidth]{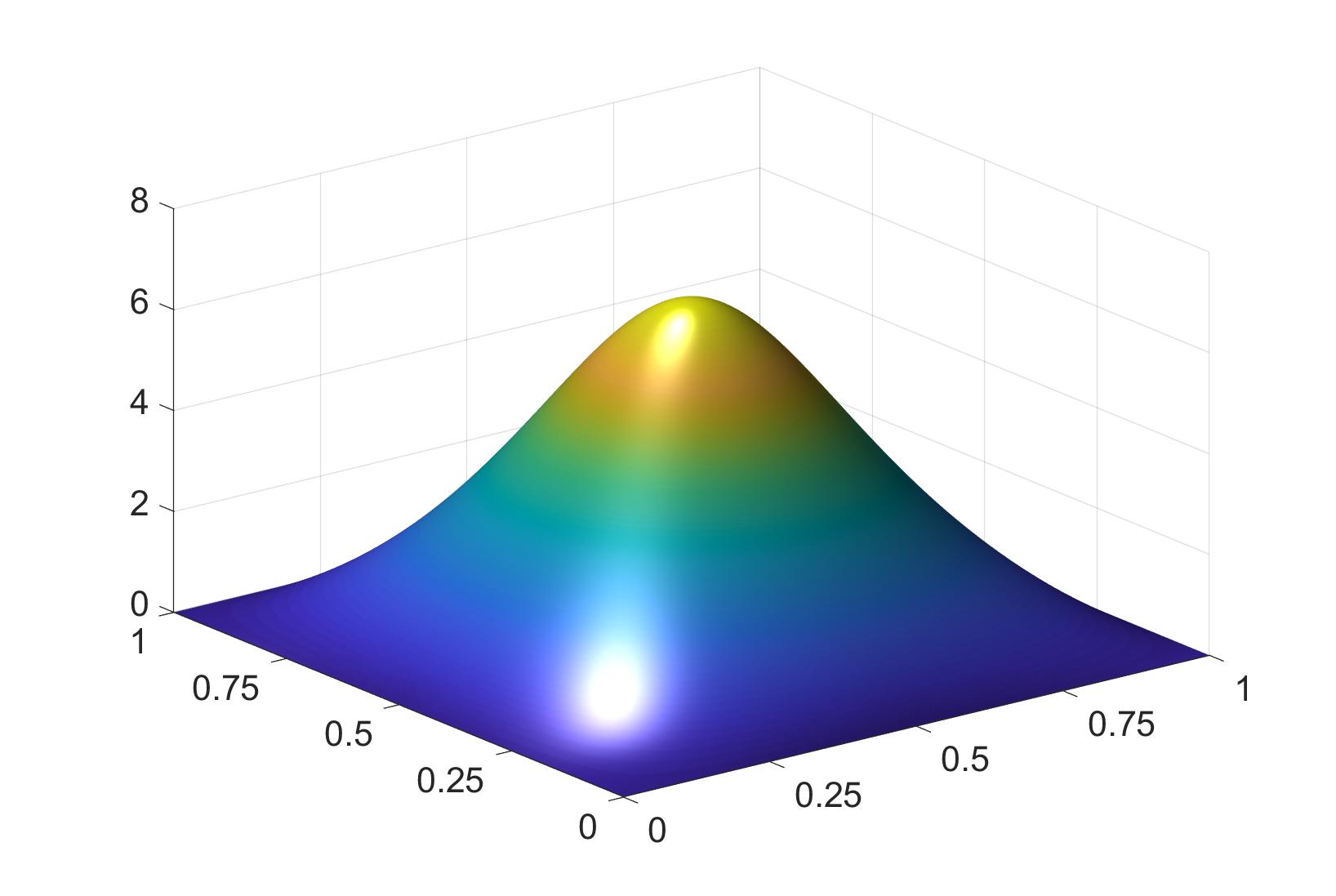}}\hfill
 \includegraphics[width=0.48\textwidth]{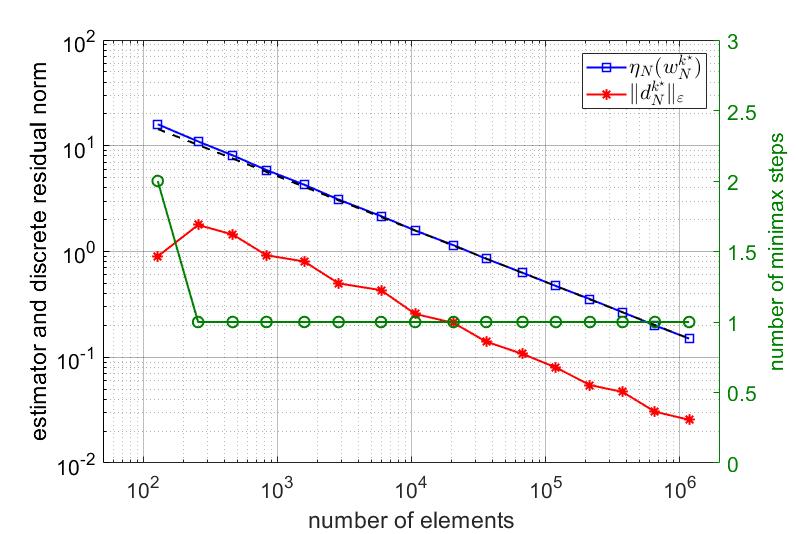} \hfill
 \caption{Experiment~\ref{sec:explaneemden}. Left: Approximated solution. Right: Convergence plots  and number of minimax steps on each finite element space.}\label{fig:LaneEmden}
\end{figure}
   
\subsubsection{Henon equation} 
\label{sec:exphenon}

Next, we focus on the problem 
\[
\begin{alignedat}{2}
-\Delta u+qu&=|\mathbf{x}|^9 u^3 & \quad & \text{in } \Omega \\
u&=0 && \text{on } \partial \Omega.
\end{alignedat}
\]
We start the numerical experiments with $L:=\{0\}$, and $v_0^0:=\sin (\pi x) \sin(\pi y)$.

\subsubsection*{Case 1: $q \equiv 0$} 
We let $\Omega=\Omega_2$. The approximated solution obtained by the LMMG Algorithm~\ref{alg:LMMG} is plotted in Figure~\ref{fig:Henon} (left). We remark that the computed solution exhibits a spike close to the corner $(-1,1)$, and coincides, up to sign, with the one from \cite[Fig.~7 (left)]{WangZhou:05}. As in the experiment before in \S\ref{sec:explaneemden}, the convergence rate of the residual indicator is optimal, see Figure~\ref{fig:Henon} (right). Moreover, we observe that the mesh is mostly refined at the location of the spike, see Figure~\ref{fig:Henon_mesh}.  

We remark that the symmetry of the initial guess with respect to the origin is not retained by the adaptive LMMG algorithm. This is likely due to the fact that local refinements of the elements are not always performed in a consistently symmetric way. Indeed, if we repeat the same experiment based on a sequence of uniform mesh refinements, then a symmetric solution is obtained, which, again up to a sign change, resembles the one from~\cite[Fig.~9 (center)]{WangZhou:05}; see Figure~\ref{fig:Henon_Uniform}. We emphasize, however, that this latter critical point features a higher energy than the former one. 

\subsubsection*{Case 2: $q\equiv 1$} 
We let $\Omega=\Omega_1$. Applying the same initial configuration as in the case $q\equiv 0$ above, a qualitatively similar solution (i.e. exhibiting one corner spike) is obtained by the LMMG algorithm, see Figure~\ref{fig:Henonq1} (left). In a subsequent experiment, we use this approximated critical point (denoted by $w_1$) to define~$L=\{w_1\}$. For this choice, the LMMG algorithm (based on the same initial guess as before) seems to generate a new solution, which we illustrate in Figure~\ref{fig:Henonq12} (left). 

\begin{figure}[ht]
\hfill \includegraphics[width=0.48\textwidth]{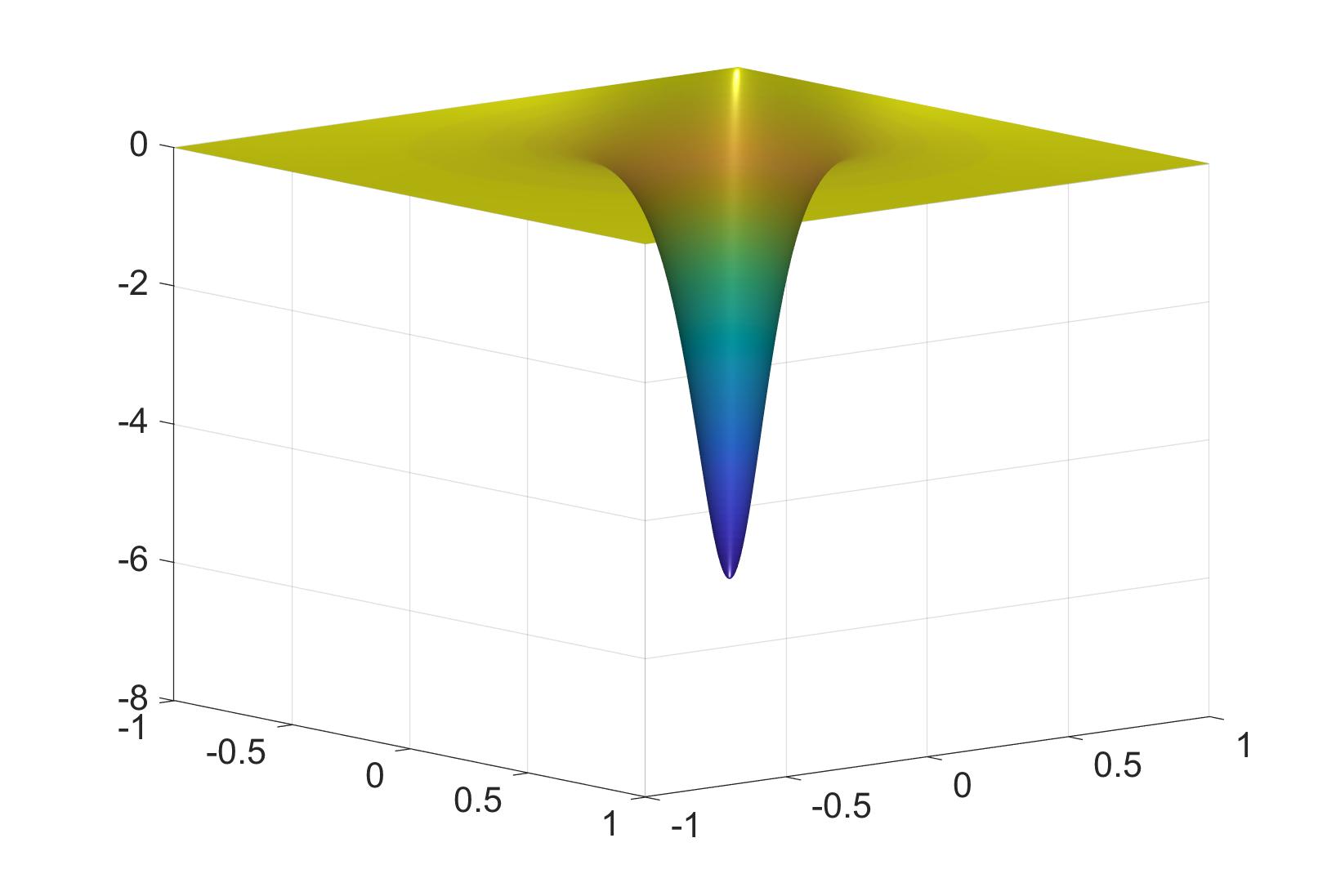}\hfill
\includegraphics[width=0.48\textwidth]{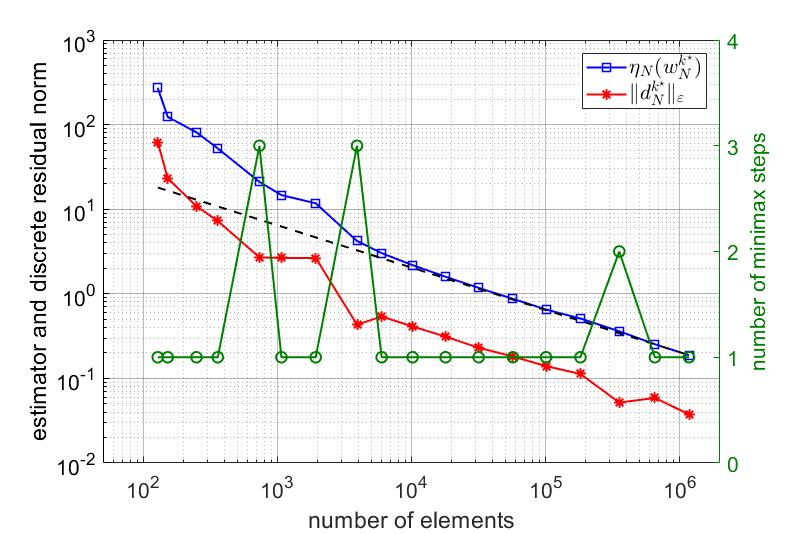} \hfill
 \caption{Experiment~\ref{sec:exphenon} with $q\equiv0$. Left: Approximated solution. Right: Convergence plots and number of minimax steps.}\label{fig:Henon}
\end{figure}

\begin{figure}[ht]
\centering
\includegraphics[width=0.5\textwidth]{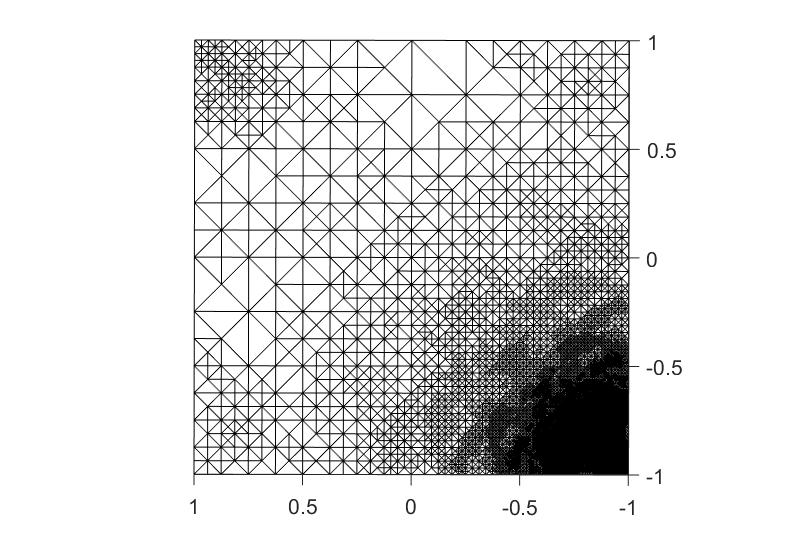}
\caption{Experiment~\ref{sec:exphenon} with $q\equiv0$: Mesh after 11 adaptive refinements.}
\label{fig:Henon_mesh}
\end{figure}

\begin{figure}[ht] 
 \hfill \includegraphics[width=0.48\textwidth]{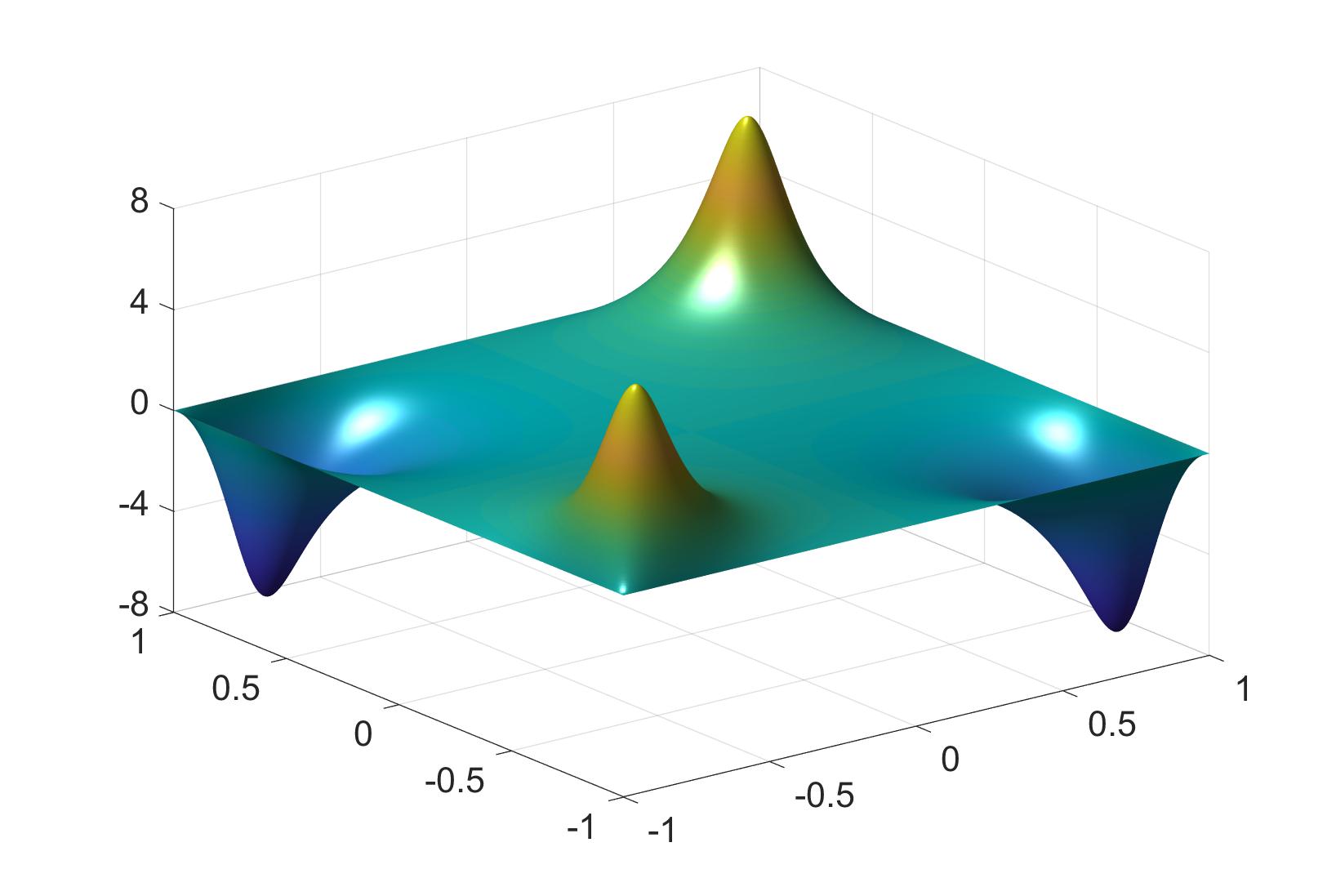} \hfill \includegraphics[width=0.48\textwidth]{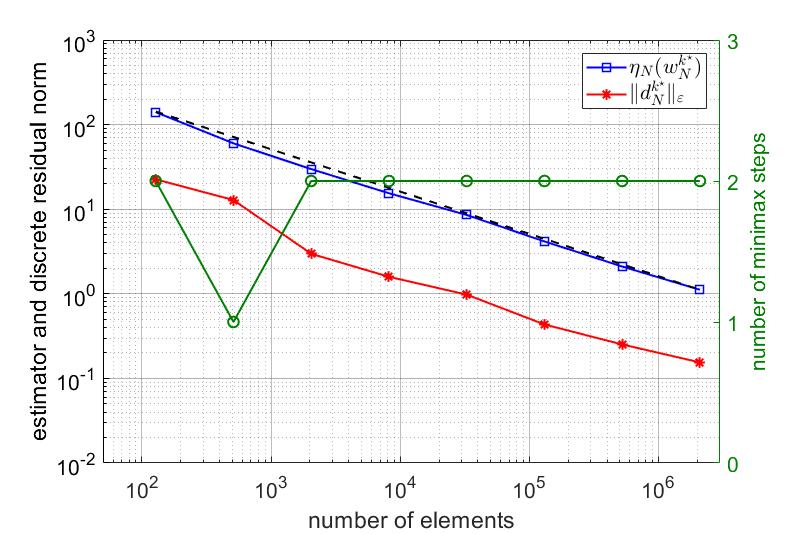}\hfill
 \caption{Experiment~\ref{sec:exphenon} with $q\equiv0$ and uniform mesh refinement. Left: Approximated solution. Right: Convergence plots and number of minimax steps.}\label{fig:Henon_Uniform}
\end{figure}

\begin{figure}[ht] 
 \hfill \includegraphics[width=0.48\textwidth]{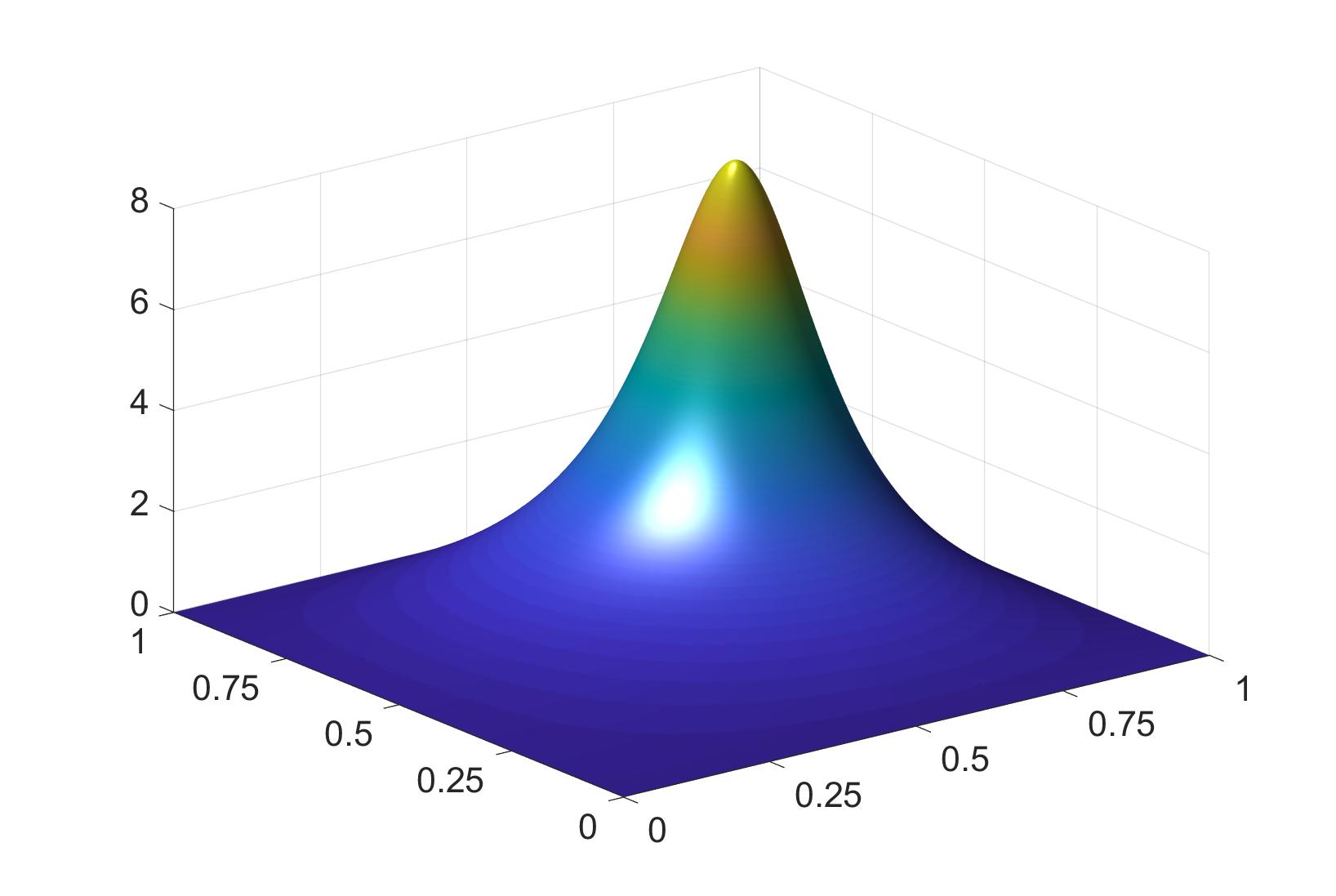}\hfill
\includegraphics[width=0.48\textwidth]{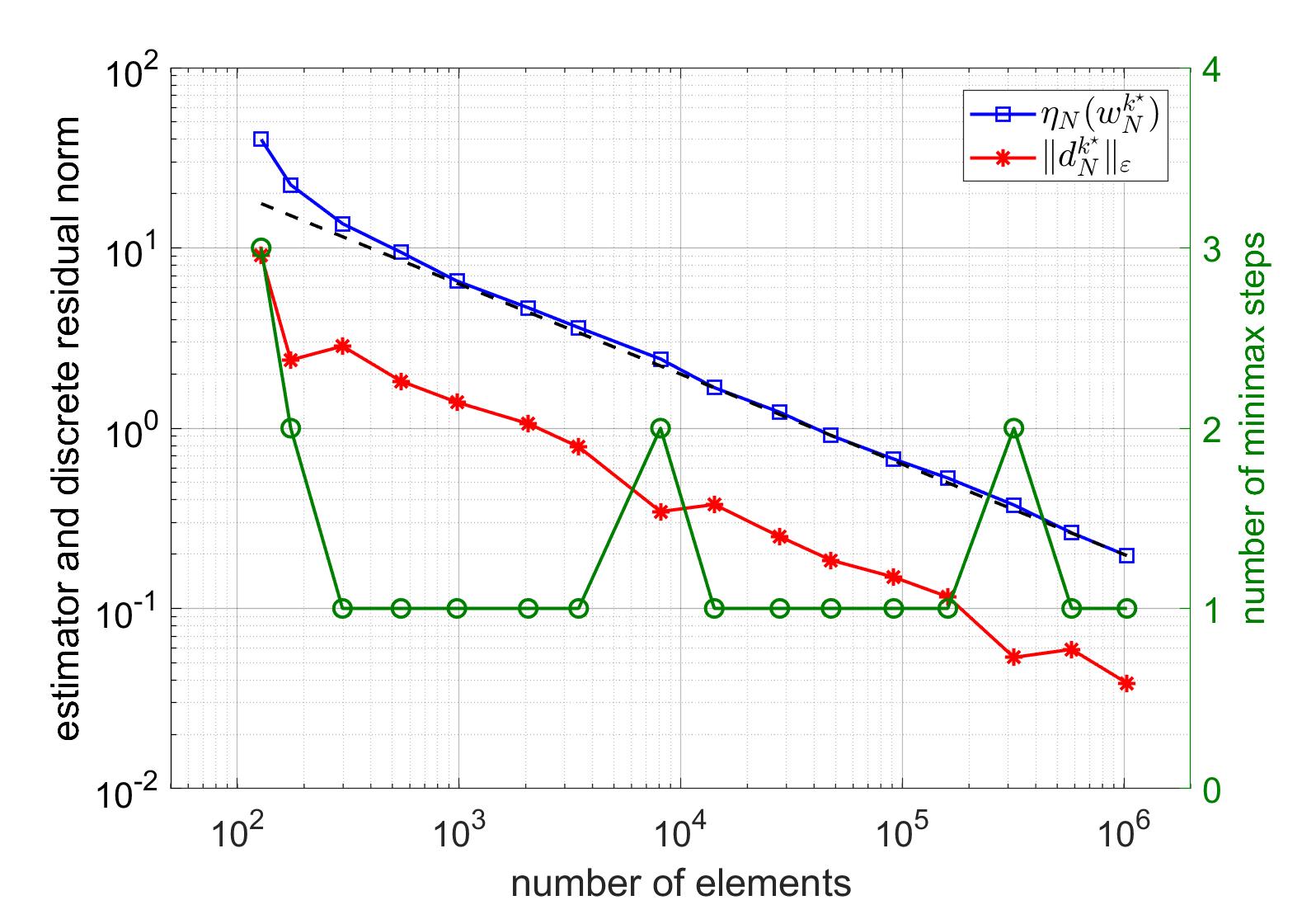}
 \caption{Experiment~\ref{sec:exphenon} with $q\equiv1$ and $L=\{0\}$. Left: Approximated solution $w_1$. Right: Convergence plots and number of minimax steps.}\label{fig:Henonq1}
\end{figure}

\begin{figure}[ht] 
 \hfill
 \includegraphics[width=0.48\textwidth]{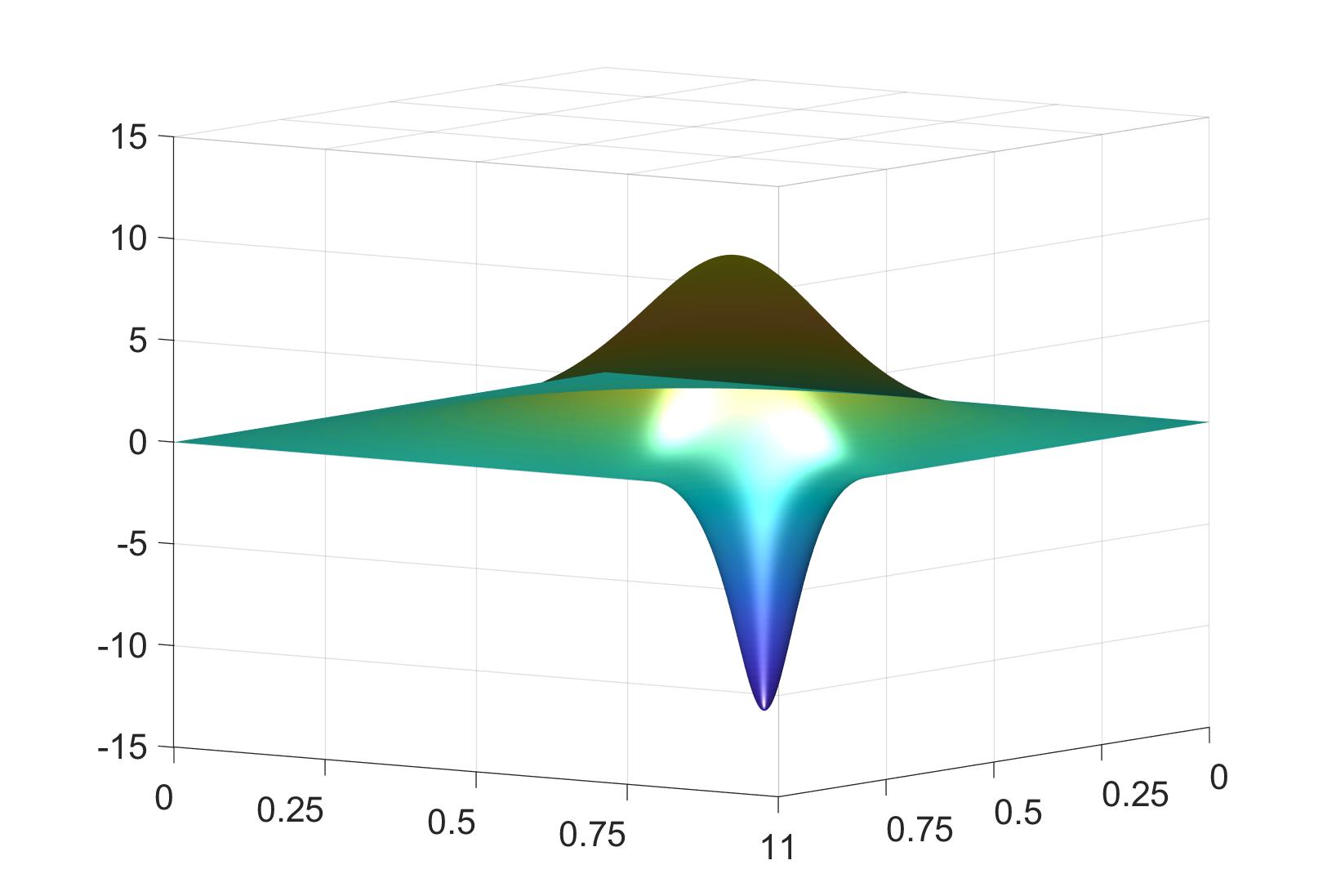} \hfill
\includegraphics[width=0.48\textwidth]{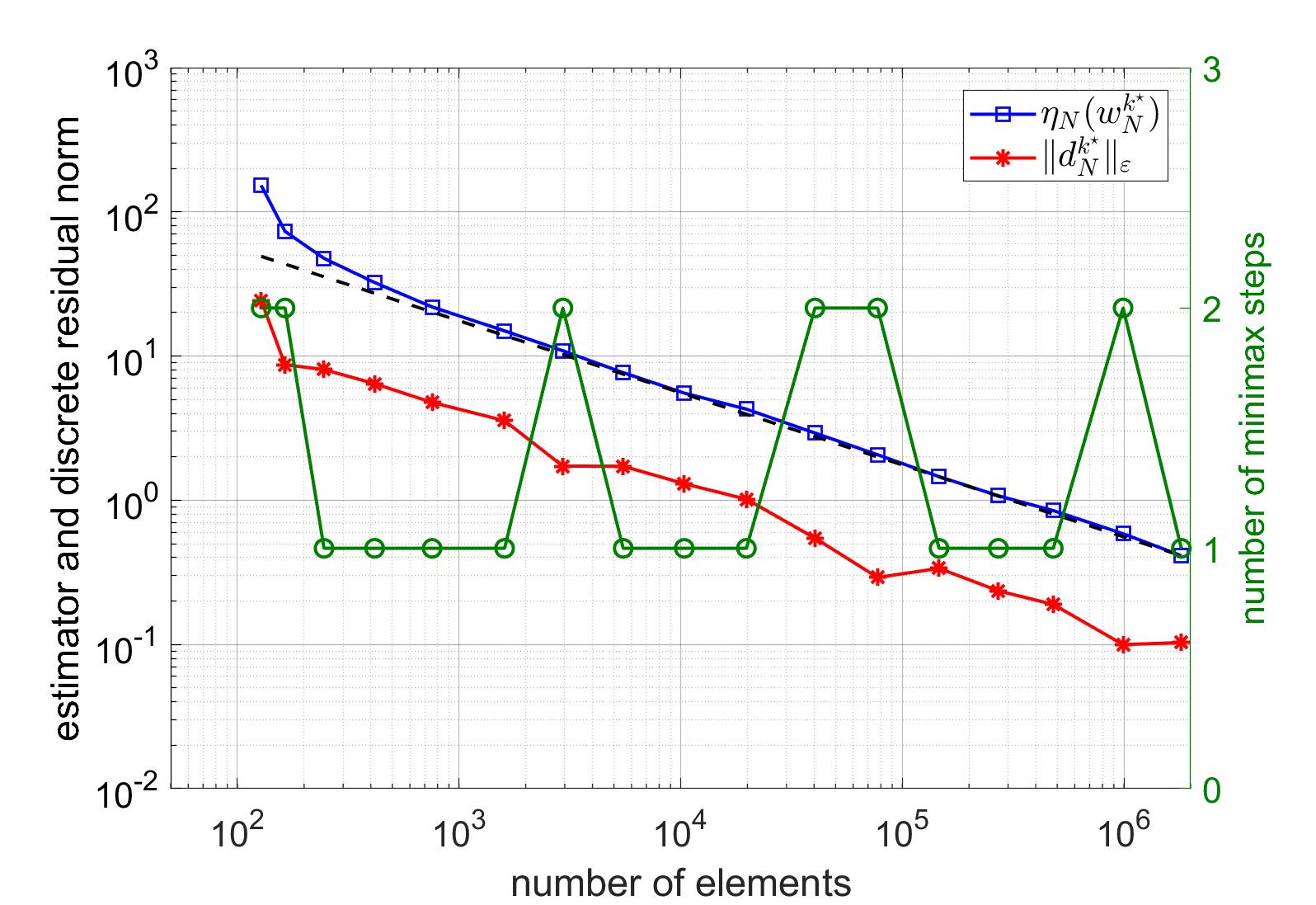} \hfill
  \caption{Experiment~\ref{sec:exphenon} with $q\equiv 1$ and $L=\{w_1\}$. Left: Approximated solution. Right: Convergence plots and number of minimax steps.}\label{fig:Henonq12}
\end{figure}

\subsubsection{A singularly perturbed Henon type equation with $q \equiv 1$} \label{sec:expsingperthenon}

We consider the problem
\[
\begin{alignedat}{2}
-\varepsilon \Delta u+u&=|\mathbf{x}|^9u^3 & \quad & \text{in } \Omega_1 \\
u&=0 && \text{on } \partial \Omega_1,
\end{alignedat}
\]
with the singular perturbation parameter $\varepsilon=0.001$; in this experiment, we set $\gamma=0.125$ and $\lambda=0.25$. The convergence of the residual estimator and the numerical solution are depicted in Figure~\ref{fig:PerturbedHenon2} (left) and~\ref{fig:PerturbedHenon2mesh} (left), respectively. We see that the spike in the corner has become much sharper, yet, the convergence rate is again asymptotically optimal. This has been accomplished by means of appropriate adaptive local element refinements near the spike, see Figure~\ref{fig:PerturbedHenon2mesh} (right). Indeed, if we employ uniform mesh refinement in Algorithm~\ref{alg:LMMG}, then the same solution is obtained, however, the convergence regime is inferior as expected, see Figure 7 (right); here, we note that the uniform mesh needs to be sufficiently fine (thereby requiring a considerably higher number of degrees of freedom) in order to be able to properly resolve the singular effects. Moreover, in comparison with the results in \S\ref{sec:exphenon} for $\varepsilon=1$, the performance for the adaptive mesh refinement strategy has not deteriorated in the singularly perturbed case. This underlines the robustness of the adaptive LMMG with respect to $\varepsilon\ll1$.

 \begin{figure}[ht] 
 \hfill
 \includegraphics[width=0.48\textwidth]{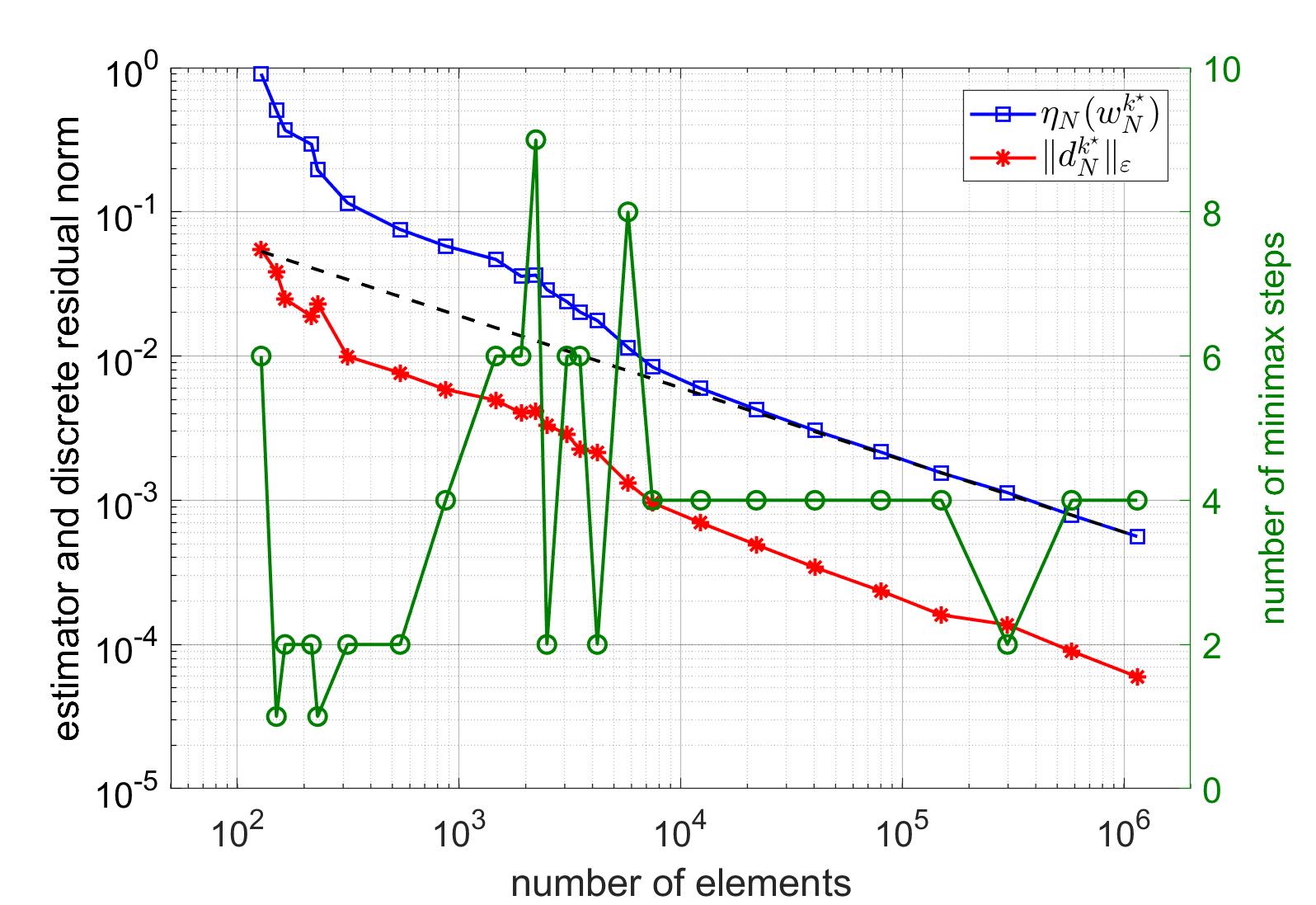}
\hfill 
 \includegraphics[width=0.48\textwidth]{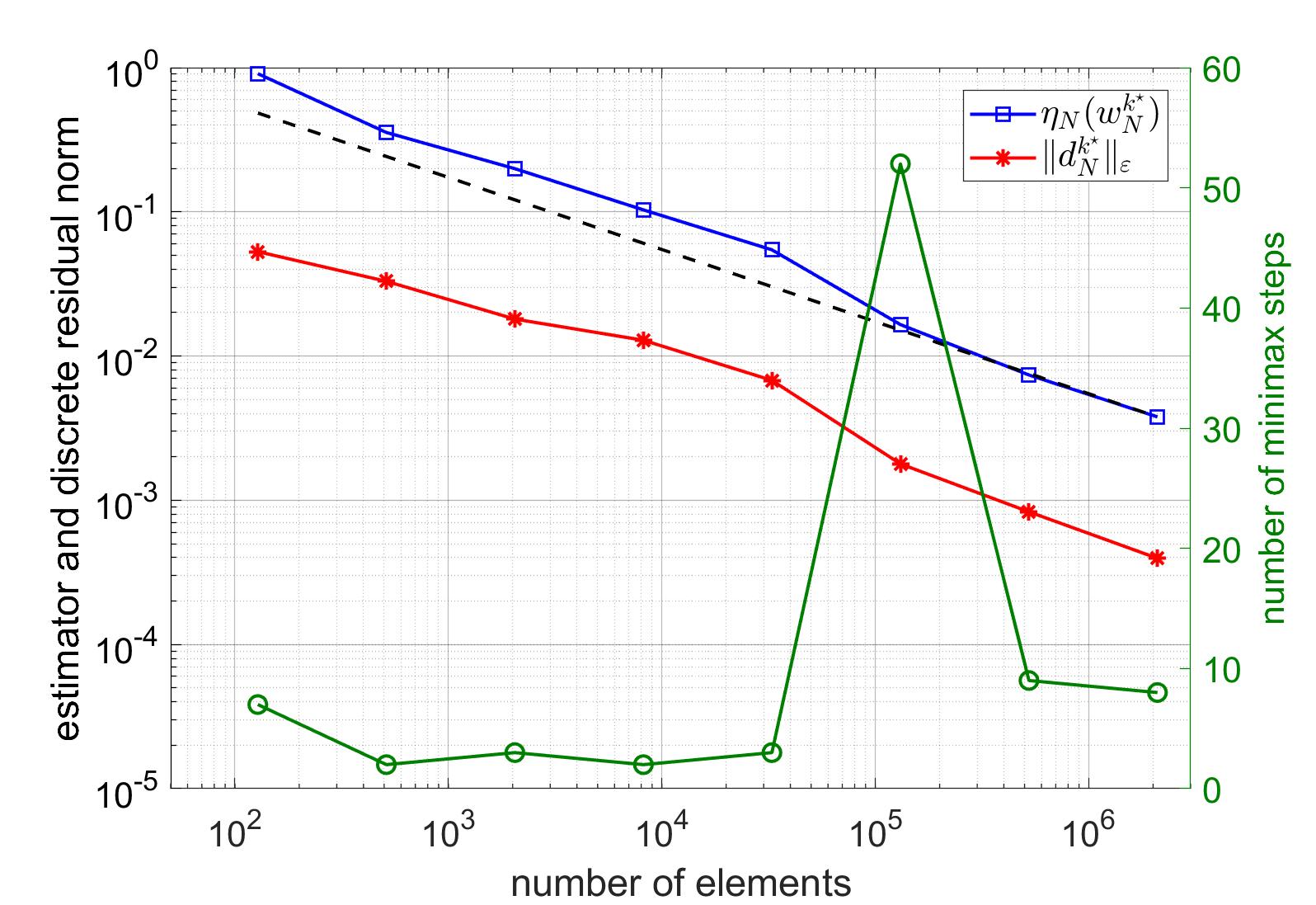}\hfill
 \caption{Experiment~\ref{sec:expsingperthenon} with $\varepsilon=10^{-3}$. Convergence plots and number of minimax steps for adaptive mesh refinement (left) and  uniform refinement (right).}\label{fig:PerturbedHenon2}
\end{figure}

 \begin{figure}[ht] 
 \hfill
 \includegraphics[width=0.48\textwidth]{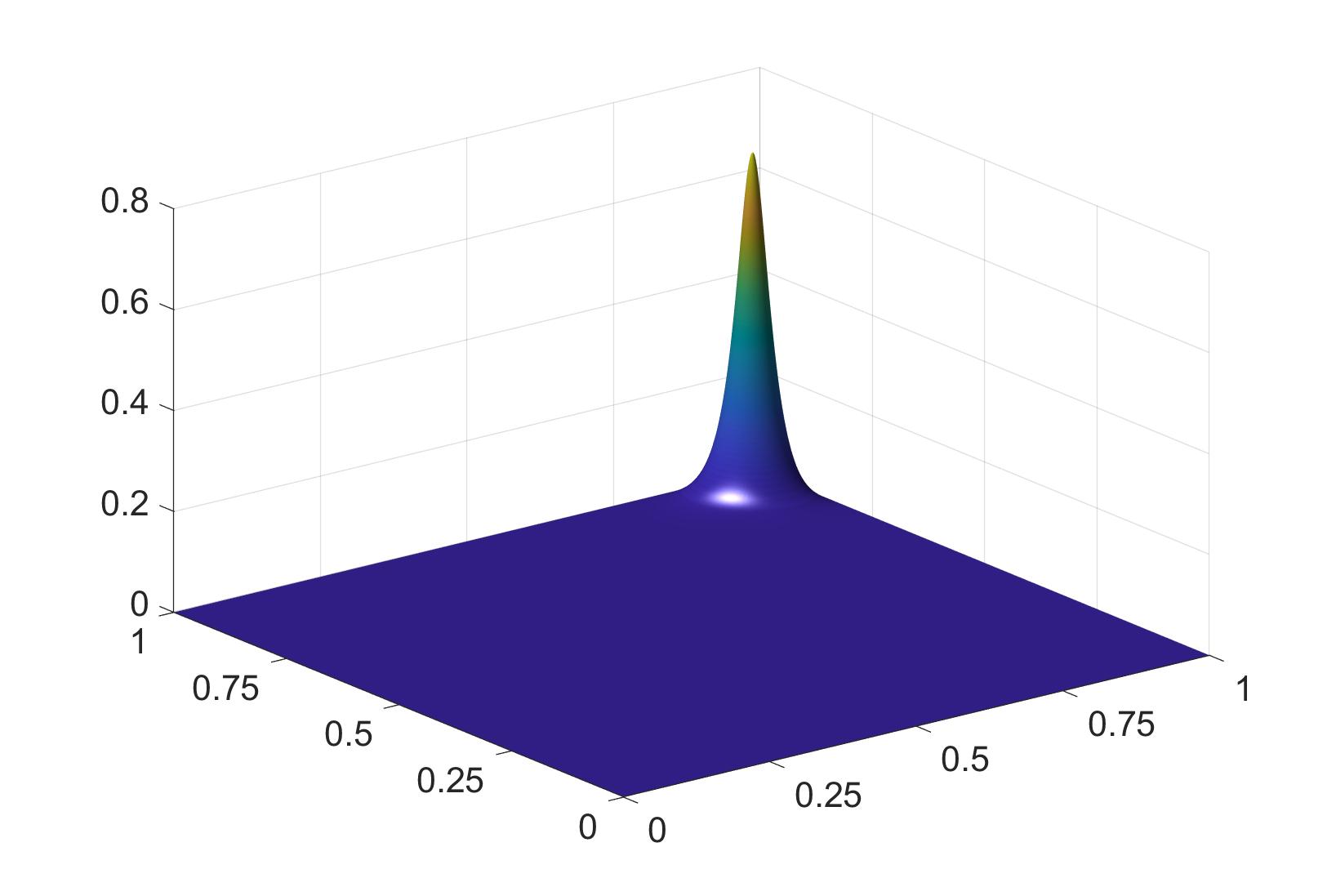}\hfill
 \includegraphics[width=0.48\textwidth]{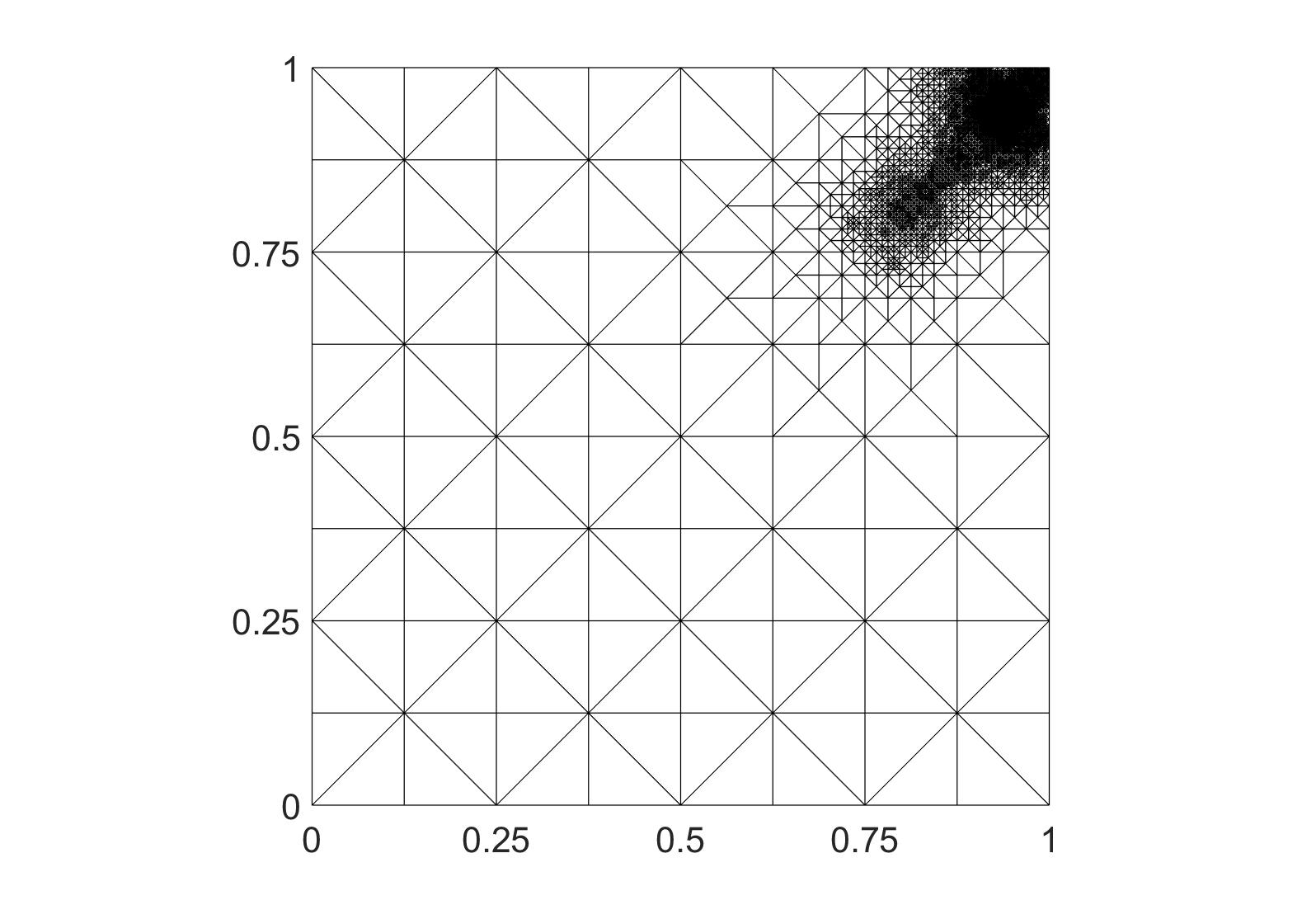}
\hfill 
 \caption{Experiment~\ref{sec:expsingperthenon} with $\varepsilon=10^{-3}$. Left: Approximated solution. Right: Adaptive mesh after 18 refinement steps.}\label{fig:PerturbedHenon2mesh}
\end{figure}

 \subsubsection{A singularly perturbed Lane-Emden type equation with $q \equiv 1$} \label{sec:expsingpertlaneemden}

We consider a further singularly perturbed problem,
\[
\begin{alignedat}{2}
-\varepsilon \Delta u+u&=u^3 & \quad & \text{in } \Omega_1 \\
u&=0 && \text{on } \partial \Omega_1,
\end{alignedat}
\]
now with $\varepsilon=10^{-8}$. Since this is a strongly perturbed problem, we expect a thin spike (or layer) in the solution. This is confirmed in Figure~\ref{fig:PerturbedLaneEmden} (left), where we observe a thin spike in the solution at the center of the domain, which is carefully resolved by adaptive mesh refinements, see Figure~\ref{fig:PerturbedLaneEmdenMesh}. Once more the convergence rate is optimal, see Figure~\ref{fig:PerturbedLaneEmden} (right), which further highlights the $\varepsilon$-robustness of the LMMG algorithm. In this experiment, due to the appearance of a thin spike in the solution, we notice again that the local mesh refinement strategy is far superior to the uniform mesh refinement, see Figure~\ref{fig:PerturbedLaneEmden}.

 \begin{figure}[ht] 
 \hfill {\includegraphics[width=0.48\textwidth]{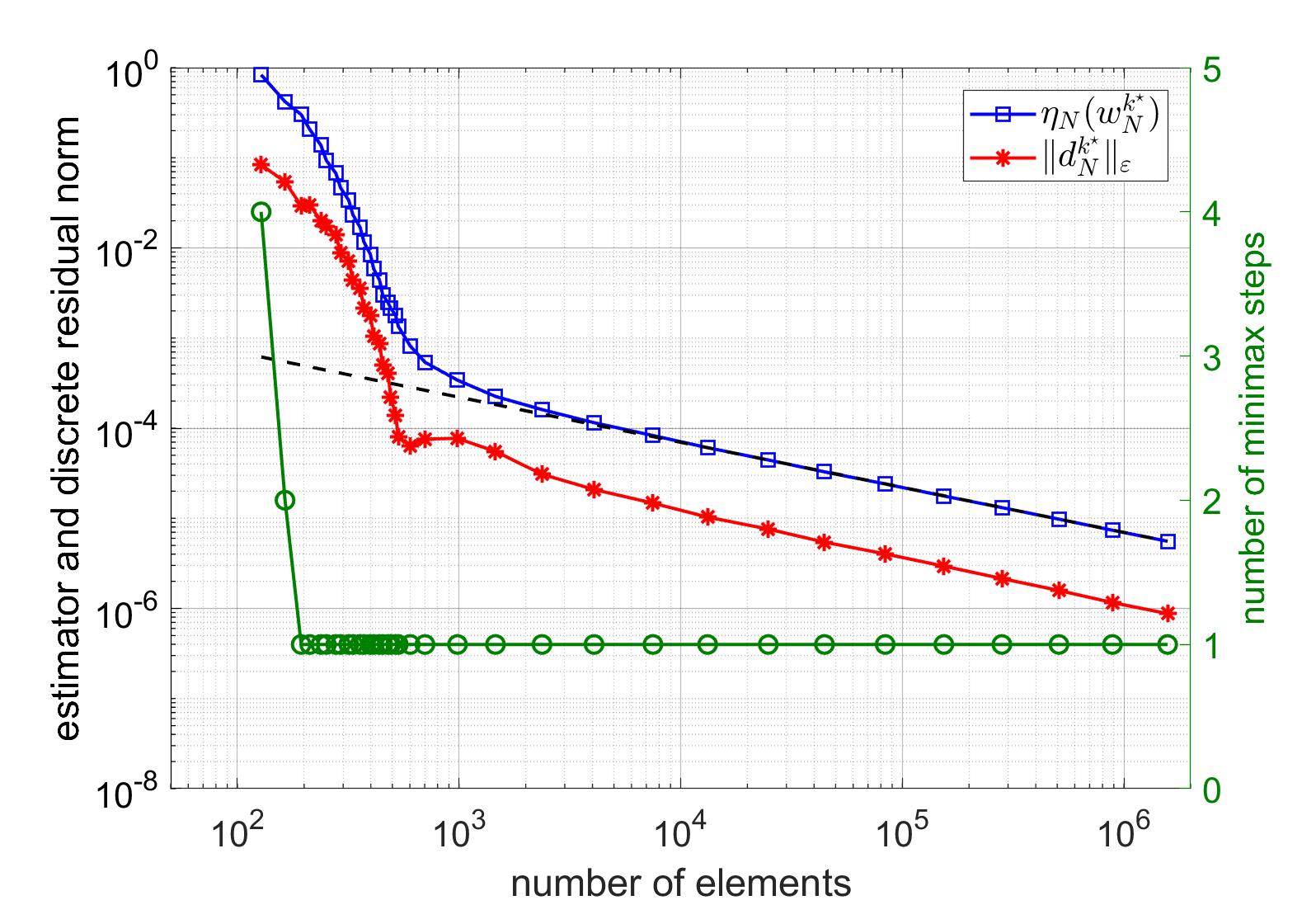}}\hfill
 {\includegraphics[width=0.48\textwidth]{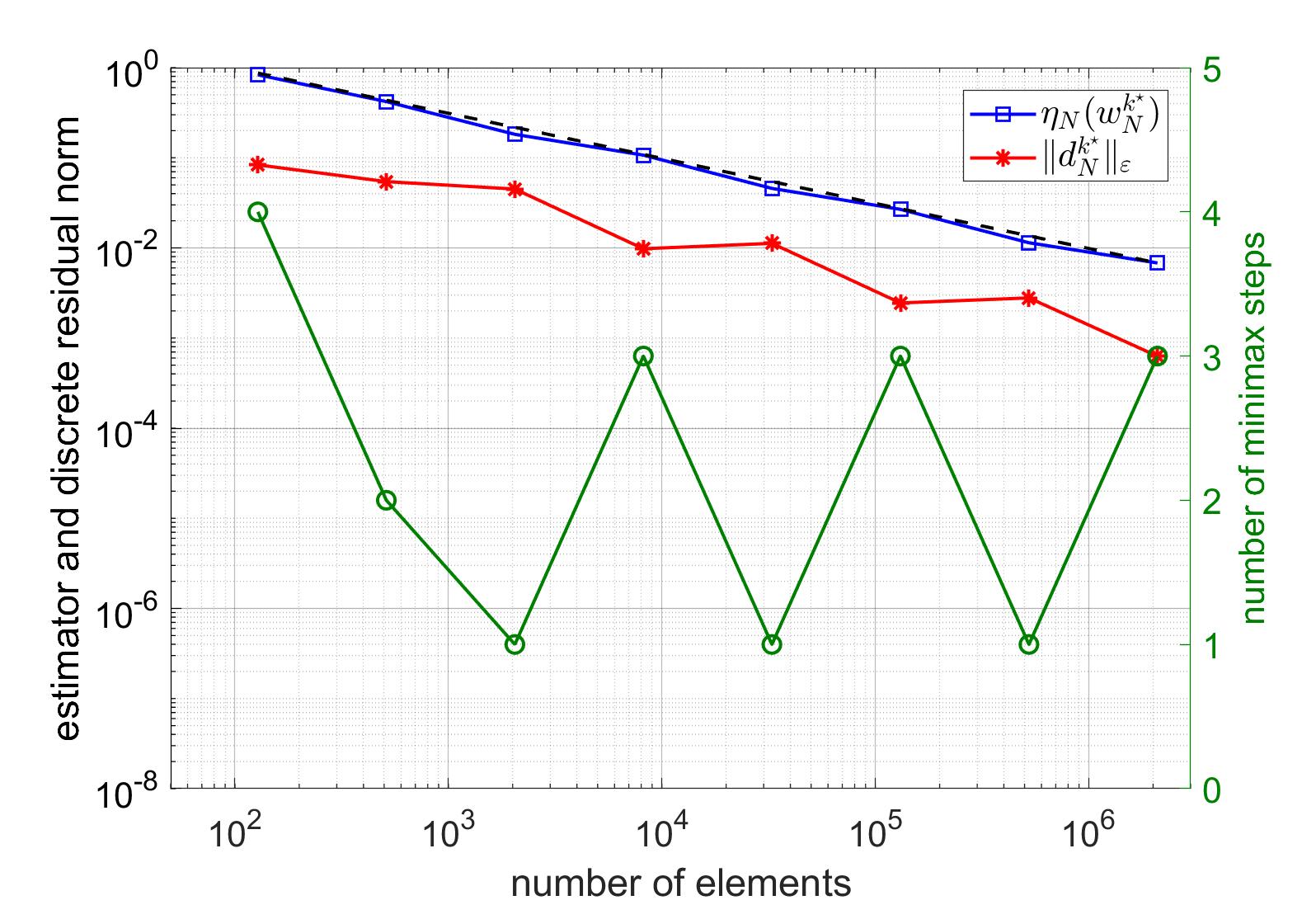}} \hfill
 \caption{Experiment~\ref{sec:expsingpertlaneemden}  with $\varepsilon=10^{-3}$. Convergence plots and number of minimax steps for adaptive mesh refinement (left) and uniform refinement (right).}\label{fig:PerturbedLaneEmden}
\end{figure}

 \begin{figure}[ht] 
 \ \hfill
 \includegraphics[width=0.48\textwidth]{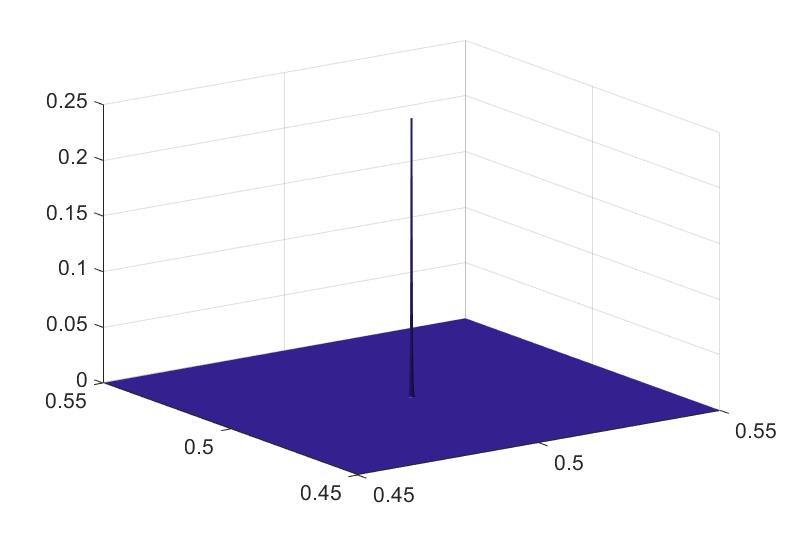}\hfill
 \includegraphics[width=0.48\textwidth]{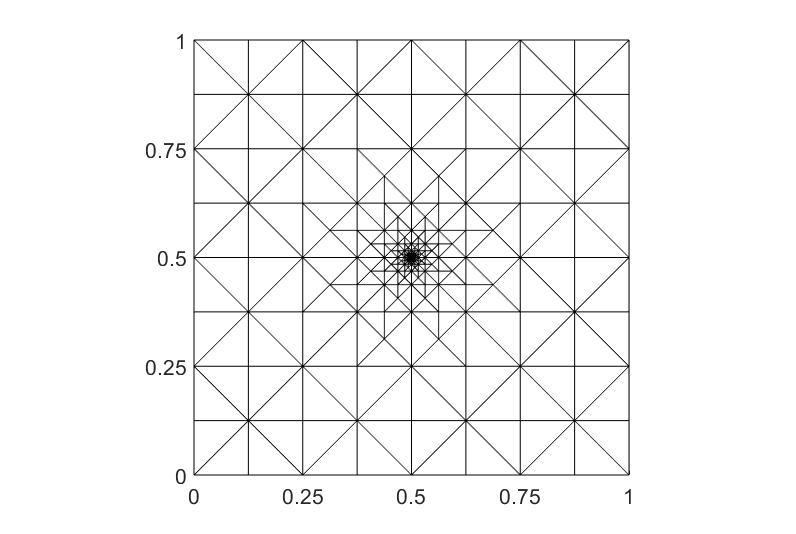}
\hfill 
 \caption{Experiment~\ref{sec:expsingpertlaneemden} with $\varepsilon=10^{-8}$. Left: Approximated solution. Right: Adaptive mesh after 25 refinement steps.}\label{fig:PerturbedLaneEmdenMesh}
\end{figure}

\section{Conclusion}

In this work, we have established an adaptive computational scheme and a convergence analysis for the numerical approximation of saddle points of non-convex energy functionals. Our approach is based on using an instantaneous interplay of the local minimax method from~\cite{LiZhou:01,LiZhou:02} and adaptive Galerkin space enrichments. Our numerical procedure was tested for a class of adaptive finite element discretizations of (singularly perturbed) semilinear elliptic PDE. The experiments presented here illustrate that the proposed LMMG algorithm is able to properly resolve thin spikes, and to achieve optimal convergence rates for the residual estimator.

\begin{appendix}

\section{Proof of Proposition~\ref{pr:app}}\label{sc:app}

We will outline the proof of Proposition~\ref{pr:app} for $n \geq 3$; the case $n=2$ is similar (and, in fact, slightly simpler). In order to do so, we will proceed along the lines of~\cite[\S4]{LiZhou:01}. We define the norm
\[
\nnn{v}^2_\varepsilon=\varepsilon\int_\Omega|\nabla v|^2\dx
+\int_\Omega qv^2\dx,\qquad v\in\HS(\Omega),
\]
which is equivalent to the norm~$\norm{\cdot}_\varepsilon$ from~\eqref{eq:epsilonnorm}. In particular, it holds that
\begin{align}\label{eq:nuconst}
\norm{v}^2_\varepsilon\le \nnn{v}^2_\varepsilon \le \cU \norm{v}^2_\varepsilon,
\end{align}
for any $v\in\HS(\Omega)$, where $\cU:=\max\{1,c_\nu\}$, cp.~\eqref{eq:qbound}.

\subsection*{Proof of (i): Uniqueness}

We show that the peak selection $p$ of $\E_\varepsilon$ with respect to $L=\{0\}$ is uniquely defined. To this end, we adapt the proof of~\cite[Lem.~2]{Ni:89} for our more general case. For any $v \in S_{\H}$, i.e. $\norm{v}_\varepsilon=1$, let 
\begin{align} \label{eq:dvdef}
   g_v(t):=\E_\varepsilon(tv)
   =\frac{t^2}{2}\nnn{v}^2_{\varepsilon}- \int_\Omega F(x,tv) \dx,
\end{align}
for $t\ge 0$. Interchanging differentiation and integration, by invoking (f1), an elementary calculation reveals that 
  \begin{align} \label{eq:gderivative}
   g_v'(t)
   =t \nnn{v}^2_{\varepsilon} -\int_\Omega f(x,tv)v \dx.
  \end{align}
Upon defining the set $\mathcal{A}_v:=\{x \in \Omega:\, v(x) \neq 0\}$, we see that $g_v'(t)=0$ if and only if 
\begin{align} \label{eq:gderequiv}
\int_{\mathcal{A}_v} \frac{f(x,t v)}{tv} v^2 \dx=\nnn{v}^2_{\varepsilon}.
\end{align}
Due to (f5) note that the left-hand side of \eqref{eq:gderequiv} is strictly increasing in $t>0$. Consequently, it exists at most one  $t^\star_v>0$ such that $g_v'(t^\star_v)=0$. It remains to establish the existence of $t^\star_v$. To this end, we examine the function $g_v(t)$ for small and large $t$ in order to show that $g_v$ changes sign. 
\begin{enumerate}[1.]
\item For any $\varepsilon'>0$, the assumptions (f2)--(f3) imply that there exists a constant $C_{\varepsilon'} >0$ such that 
\[
 f(x,t) \leq 2\varepsilon' |t| + C_{\varepsilon'}(s+1) |t|^s\qquad \forall t \in \mathbb{R}.
\]
Therefore, upon integration, we deduce that
\begin{align} \label{eq:Fepsiloni}
 F(x,t) \leq \varepsilon' t^2 + C_{\varepsilon'} |t|^{s+1} \qquad \forall t \in \mathbb{R}.
\end{align}
Hence,
\begin{equation*}
\left|\int_\Omega F(x,vt)\dx\right|
\le \varepsilon't^2\int_\Omega v^2\dx+C_{\varepsilon'}t^{s+1}\int_\Omega|v|^{s+1}\dx,\qquad t\ge 0.
\end{equation*}
Notice that $\HS(\Omega)$ equipped with the standard $\HS(\Omega)$-norm is continuously embedded in $\L^{s+1}(\Omega)$ for $1\le s+1\le\nicefrac{2n}{(n-2)}$; equivalently, $0\le s\le\nicefrac{(n+2)}{(n-2)}$, which is satisfied for $n\ge 3$ in view of condition (f2). Thus, since $\norm{v}_\varepsilon=1$, and because of the equivalence of the norm $\norm{\cdot}_{\varepsilon}$ and the standard $\HS(\Omega)$-norm, there are constants $C_1,C_2 \ge0$ (depending on~$\varepsilon$) such that 
\begin{align}\label{eq:iF}
\left|\int_\Omega F(x,vt)\dx\right|
&\le \varepsilon'C_1t^2+C_2C_{\varepsilon'}t^{s+1},
\end{align}
for $t\ge 0$. Moreover, comparing the lower and upper bounds in~\eqref{eq:Fepsiloni} and \eqref{eq:f4cons}, respectively, and recalling (f4), we observe that $s+1\ge\mu>2$. Since $\varepsilon'>0$ is chosen arbitrarily, it follows from~\eqref{eq:dvdef} and~\eqref{eq:iF} that 
\[
g_v(t)=\frac{t^2}{2}\nnn{v}^2_{\varepsilon}+o(t^2),\qquad t\searrow 0,
\]
with $o(t^2)$ independent of $v \in S_\X$. Hence, invoking~\eqref{eq:nuconst}, we find that 
\begin{equation}\label{eq:gvol}
g_v(t)\geq \frac{t^2}{2}+o(t^2),\qquad t\searrow 0,
\end{equation}
uniformly in $v \in S_\X$. Consequently, there is a constant $\rho>0$ such that $\sup_{t>0} g_v(t)\geq \rho>0$ for all $v \in S_{\H}$.

\item Applying the bounds~\eqref{eq:f4cons} and \eqref{eq:nuconst} in \eqref{eq:dvdef}, results in
\begin{align} \label{eq:gvupperbound}
 g_v(t) &\leq \frac{t^2}{2}\cU  + a_4 |\Omega| - a_3t^\mu \int_\Omega |v|^\mu \dx,
\end{align}
where $|\Omega|$ signifies the volume of $\Omega$. Since $\mu>2$, the right-hand side of the above estimate tends to $-\infty$ for $t \to \infty$; thus, the same applies for $g_v(t)$. 
\end{enumerate}
Recalling~\eqref{eq:F}, we note that $g_v(0)=0$. Hence from the two steps above, we conclude that $g_v(t)$ attains at least one maximum for $t>0$.  In summary, we have shown that there exists a unique maximum of $g_v(t)$ in $t \geq 0$ for any $v \in S_{\H}$. Therefore, the unique peak selection $p:\,S_{\H} \to {\H}$ is given by $p(v)=t^\star_vv$, where $t^\star_v > 0$ is the unique solution of $g_v'(t)=0$, cf.~\eqref{eq:gderivative}, for fixed $v \in S_{\H}$.

\subsection*{Proof of (i): Continuity}
Next, by proceeding along the lines of the proof of~\cite[Thm.~4.3.]{LiZhou:01}, we show that the peak selection $p$ is continuous. To this end, we define the map $G:\,S_\H \times \mathbb{R}^+ \to \mathbb{R}$ by $G(v,s):=g_v'(s)$. By assumption (f6), we conclude that $G(v,s)$ is continuously differentiable in $s$, and we have 
\begin{align*}
\frac{\partial G}{\partial s}(v,s)= \nnn{v}_\varepsilon^2-\int_\Omega v^2 \frac{\partial f}{\partial t}(x,sv) \dx. 
\end{align*}
Furthermore, due to (f5), we observe that $f_t(x,t)>f(x,t)t^{-1}$, for any $x\in\Omega$ and $t\in \mathbb{R} \setminus \{0\}$. Fix an arbitrary $v_0 \in S_\H$ and set $s_0:=t_{v_0}^\star>0$, i.e.~$G(v_0,s_0)=g_{v_0}'(s_0)=0$. Then, recalling~\eqref{eq:gderivative}, we deduce
\begin{align*}
0&=G(v_0,s_0)
>s_0 \nnn{v_0}_\varepsilon^2 -  \int_\Omega s_0 v_0^2 f_t(x,s_0 v_0) \dx=s_0 \frac{\partial G}{\partial s}(v_0,s_0).  
\end{align*}
Since $s_0>0$, we obtain that $\frac{\partial G}{\partial s}(v_0,s_0)<0$. Thus, by the implicit function theorem, there exists an open neighbourhood $\mathcal{U}(v_0)\subset S_{\H}$ of $v_0$, and a \emph{continuous} map $\sigma:\,\mathcal{U}(v_0) \to \mathbb{R}^+$ such that $G(v,\sigma(v))=0$ for all $v \in \mathcal{U}(v_0)$; see, e.g.,~\cite[Thm.~4.E]{Zeidler:95}. By uniqueness it follows that $\sigma(v)=t_v^\star$, and the peak selection $p(v)=t_v^\star v=\sigma(v)v$, for $v\in\mathcal{U}(v_0)$, is continuous. Since this holds for any $v_0 \in S_\X$, the peak selection mapping $p:S_\X \to \X$ is globally continuous.

\subsection*{Proof of (ii)}

From the proof of (i) it follows that 
\begin{align} \label{eq:functionalpositive}
 \Ee(p(v))=\max_{t > 0} \E_\varepsilon(tv)=\max_{t>0}g_v(t)\geq \rho>0 \qquad \forall v \in S_{\H},
\end{align}
which proves (ii). 

\subsection*{Proof of (iii)}
Recalling~\eqref{eq:functionalpositive} and~\eqref{eq:gvol}, which is uniform with respect to $v \in S_{\H}$, it follows that there exists $\alpha>0$ such that
\begin{align*}
\argmax_{t>0} \E_\varepsilon(tv)
=\argmax_{t>0} g_v(t) \geq \alpha \qquad \forall v \in S_{\H}.
\end{align*}
Thus, for $L=\{0\}$, we deduce
\begin{align*}
\inf_{y\in L}\norm{p(v)-y}_{\varepsilon}=\norm{p(v)}_\varepsilon=t^\star_v=\argmax_{t>0} \E_\varepsilon(tv) \geq \alpha >0 \qquad \forall v \in S_{\H}.
\end{align*}
This shows (iii).

\subsection*{Proof of (iv)}

We first show that the peak selection $p$ constructed above is bounded, i.e.~there exists $\beta>0$ such that
  \begin{align} \label{eq:pbounded}
   \norm{p(v)}_\varepsilon \leq \beta \qquad \forall v \in S_{\H}.
  \end{align}
If not then there is a sequence $\{p(v^\ell)\}_\ell$ with $t^\star_{v^\ell}=\norm{p(v^\ell)}_\varepsilon \to \infty$ for $\ell \to \infty$. Applying \eqref{eq:functionalpositive} and \eqref{eq:gvupperbound} leads to
\begin{align*}
  0 < \rho  
  \le \E_\varepsilon(p(v^\ell))
  =\E_\varepsilon(t^\star_{v^\ell}v^\ell)
  = g_{v^\ell}(t^\star_{v^\ell})
  \le \frac{\cU}{2}\left(t^\star_{v^\ell}\right)^2
  + a_4 |\Omega| - a_3\left(t^\star_{v^\ell}\right)^\mu \int_\Omega \left|v^\ell\right|^\mu \dx.
\end{align*}
Since $\mu >2$, the right-hand side tends to $-\infty$ for $t^\star_{v^\ell} \to \infty$, which leads to the desired contradiction. Thus, $p$ is bounded.  

Using \eqref{eq:Jprime} and the Cauchy-Schwarz inequality, for $u,v\in\HS(\Omega)$, we find that 
\begin{align*}
 \dprod{\E_\varepsilon'(w),u} &\leq \nnn{w}_\varepsilon \nnn{u}_\varepsilon +\int_\Omega |f(x,w)||u| \dx.
\end{align*}
Invoking (f2), this leads to 
\begin{align*}
 \dprod{\E_\varepsilon'(w),u} &\leq \nnn{w}_\varepsilon \nnn{u}_\varepsilon+a_1 \int_\Omega |u| \dx + a_2 \int_\Omega |w|^s|u| \dx \\
 & \leq \nnn{w}_\varepsilon \nnn{u}_\varepsilon + a_1 \norm{u}_{\L^1(\Omega)}+a_2 \int_\Omega |w|^s|u| \dx.
\end{align*}
Let $\p=\nicefrac{2n}{(n-2)}>2$, and $\q=\nicefrac{\p}{(\p-1)}<\p$ its conjugate, i.e. $\nicefrac{1}{\p}+\nicefrac{1}{\q}=1$. Then, H\"older's inequality yields
\begin{align*}
\dprod{\E_\varepsilon'(w),u} & \leq \nnn{w}_\varepsilon \nnn{u}_\varepsilon + a_1|\Omega|^{\nicefrac12} \norm{u}_{\L^2(\Omega)}+a_2 \norm{w}_{\L^{s\q}(\Omega)}^s \norm{u}_{\L^{\p}(\Omega)}. 
\end{align*}
From (f2) we observe that $s<\p-1$, and, thus, $s\q< \p$. Furthermore, we note that the embedding $\HS(\Omega) \hookrightarrow \L^{\r}(\Omega)$ is continuous for $1\le \r \le \p$. Hence, we deduce the bound
\begin{align*}
\dprod{\E_\varepsilon'(w),u} & \leq \nnn{w}_\varepsilon \nnn{u}_\varepsilon+C_1 \norm{u}_\varepsilon+C_2 \norm{w}_{\varepsilon}^s \norm{u}_{\varepsilon},
\end{align*}
for some constants $C_1,C_2>0$. Then, from~\eqref{eq:nuconst}, we further obtain
\begin{align*}
\dprod{\E_\varepsilon'(w),u} & \leq \cU \norm{w}_\varepsilon \norm{u}_\varepsilon+C_1 \norm{u}_\varepsilon+C_2 \norm{w}_{\varepsilon}^s \norm{u}_{\varepsilon}.
\end{align*}
Consequently, recalling~\eqref{eq:pbounded}, it follows that
\begin{align*}
\norm{\E_\varepsilon'(p(v))}_{\H^\star}
&=\sup_{u \in S_{\H}}\dprod{\E_\varepsilon'(p(v)),u} 
\leq \cU \norm{p(v)}_\varepsilon+C_1+C_2 \norm{p(v)}_\varepsilon^s
\leq \cU \beta+C_1+C_2 \beta^s,
\end{align*}
for all $v \in S_{\H}$. This proves (iv) with $\gamma=\cU \beta+C_1+C_2 \beta^s< \infty$. 
 
\end{appendix}

\bibliographystyle{amsplain}
\bibliography{references}
\end{document}